\documentclass[11pt]{amsart}
\usepackage[utf8]{inputenc}
\usepackage{amsmath,amssymb, mathrsfs}
\usepackage{bbm}
\usepackage{graphicx}

\usepackage[colorlinks=true, allcolors=BlueViolet, pagebackref=true]{hyperref}
\renewcommand*\backref[1]{\ifx#1\relax \else (Cited on p.#1) \fi}
\usepackage{cleveref}
\crefname{equation}{Equation}{}
\crefname{figure}{Figure}{Figures}
\crefname{question}{Question}{Question}
\crefname{section}{Section}{Sections}
\crefname{subsection}{Subsection}{Subsections}
\crefname{lemma}{Lemma}{Lemmas}
\crefname{proposition}{Proposition}{Propositions}
\crefname{theorem}{Theorem}{Theorems}
\crefname{innercustomthm}{Theorem}{Theorems}
\crefname{mainthm}{Theorem}{Theorems}
\crefname{corollary}{Corollary}{Corollaries}
\crefname{definition}{Definition}{Definitions}
\crefname{remark}{Remark}{Remarks}
\crefname{proposition}{Proposition}{Proposition}
\crefname{corollary}{Corollary}{Corollaries}
\crefname{example}{Example}{Examples}
\crefname{conjecture}{Conjecture}{Conjectures}
\crefname{enumi}{}{}

\usepackage{marvosym}
\usepackage{wasysym}%

\usepackage{bbold}
\usepackage{centernot}
\usepackage{cite}
\usepackage[dvipsnames]{xcolor}

\usepackage{tikz}
\usetikzlibrary{arrows}
\usetikzlibrary{patterns}
\usetikzlibrary{shapes}
\usepackage{bbm}
\usepackage{enumerate}
\usepackage[shortlabels]{enumitem}
\usepackage{tikz-cd}
\usetikzlibrary{cd}

\newcommand{\R}{\mathbb{R}}
\newcommand{\N}{\mathbb{N}}

\newcommand{\de}{\partial}

\def\randin{%
\mathchoice%
{\raisebox{-.35ex}{$\displaystyle{^\subset}$}\mkern-11.5mu\raisebox{+.45ex}{$\displaystyle{_\subset}$}} {\mkern+1mu\raisebox{-.27ex}{$\textstyle{^\subset}$}\mkern-11.7mu\raisebox{+.45ex}{$\textstyle{_\subset}$}} {\raisebox{.35ex}{$\scriptstyle\subset$}\mkern-14mu\raisebox{-.15ex}{$\scriptstyle\subset$}} {\raisebox{.3ex}{$\scriptscriptstyle\subset$}\mkern-13.5mu\raisebox{-.10ex}{$\scriptscriptstyle\subset$}}
}

\newcommand{\maschera}{\textcolor{white}{\scalebox{0.3}{$\blacktriangle$}}} \def\FlatOmega{%
\mathchoice{%
\displaystyle{\Omega}\mkern-14mu\raisebox{+.166ex}{$\displaystyle{\maschera}$}
\mkern+7mu\raisebox{+.166ex}{$\displaystyle{\maschera}$}}{%
\hbox{$\textstyle{\Omega}$}\mkern-14mu\raisebox{+.166ex}{\hbox{$\textstyle{\maschera}$}}
\mkern+7mu\raisebox{+.166ex}{\hbox{$\textstyle{\maschera}$}}}{
\scriptstyle{\Omega}\mkern-14mu\raisebox{+.13ex}{$\scriptstyle{\maschera}$} \mkern+5mu\raisebox{+.13ex}{$\scriptstyle{\maschera}$}}{%
\scriptscriptstyle{\Omega}\mkern-14mu\raisebox{+.13ex}{$\scriptscriptstyle{\maschera}$} \mkern+5mu\raisebox{+.13ex}{$\scriptscriptstyle{\maschera}$}}}
\newcommand{\scaledFlatOmega}{{\scalebox{0.8}{$_{\FlatOmega}$}}} \def\randto{%
\mathchoice{%
\raisebox{-.101ex}{$\displaystyle{-}$}\mkern-4.4mu\raisebox{.729ex}{$\displaystyle{\scaledFlatOmega}$}
\mkern-5.2mu\raisebox{-.101ex}{$\displaystyle\to$}}{%
\raisebox{-.101ex}{\hbox{$\textstyle{-}$}}\mkern-4.4mu\raisebox{.729ex}{\hbox{$\textstyle{\scaledFlatOmega}$}}
\mkern-5.2mu\raisebox{-.101ex}{\hbox{$\textstyle\to$}}}{%
\raisebox{-.101ex}{$\scriptstyle{-}$}\mkern-4.4mu\raisebox{.729ex}{$\scriptstyle{\scaledFlatOmega}$}
\mkern-5.2mu\raisebox{-.101ex}{$\scriptstyle\to$}}{%
\raisebox{-.101ex}{$\scriptscriptstyle{-}$}\mkern-4.4mu\raisebox{.729ex}{$\scriptscriptstyle{\scaledFlatOmega}$}
\mkern-5.2mu\raisebox{-.101ex}{$\scriptscriptstyle\to$}}}

\newcommand{\kop}{\left\{}
\newcommand{\pok}{\right\}}
\newcommand{\tyu}{\left(}
\newcommand{\uyt}{\right)}
\newcommand{\qwe}{\left[}
\newcommand{\ewq}{\right]}
\makeatother
\renewcommand{\a}{\alpha}

\newcommand{\f}{\varphi}
\newcommand{\e}{\varepsilon}

\renewcommand{\S}{\mathbb S}

\usepackage{pifont}

\makeatletter
\newcommand{\tpitchfork}{%
  \raise-0.1ex\vbox{
    \baselineskip\z@skip
    \lineskip-.52ex
    \lineskiplimit\maxdimen
    \m@th
    \ialign{##\crcr\hidewidth\smash{$-$}\hidewidth\crcr$\pitchfork$\crcr}
  }%
}
\makeatother

\newcommand{\vol}{\mathrm{vol}}

\newcommand{\PP}{\mathbb{P}}

\newcommand{\tr}{\mathrm{tr}}

\newcommand{\m}[1]{\mathcal{#1}}

\newcommand{\Prob}{(\Omega, \mathfrak{S},\P)}
\renewcommand{\P}{\mathbb{P}}
\newcommand{\E}{\mathbb{E}}

\newcommand{\mC}{\mathcal{C}}

\newcommand{\scal}{\mathrm{scal}}

\usepackage{hyperref}
\hypersetup{
	colorlinks=true,       
	linkcolor=blue,         
	citecolor=blue}

 \usepackage[usenames,dvipsnames]{pstricks}
 \usepackage{epsfig}
 \usepackage{pst-grad} 
 \usepackage{pst-plot} 

\newtheorem{thm}{Theorem}
\newtheorem{lemma}[thm]{Lemma}
\newtheorem{cor}[thm]{Corollary}
\newtheorem{prop}[thm]{Proposition}

\theoremstyle{definition}

\newtheorem{definition}[thm]{Definition}
\newtheorem{remark}[thm]{Remark}

\newcommand{\be}{\begin{equation}}
\newcommand{\ee}{\end{equation}}

\newcommand{\bega}{\begin{equation}\begin{aligned}}
\newcommand{\eega}{\end{aligned}\end{equation}}

\usepackage{mathtools} 

\numberwithin{equation}{section}

\title{Critical Points of Chi-Fields}
\author{Domenico Marinucci and Michele Stecconi}

\date{September 2024}
\begin{document}
\begin{abstract}
We give here a semi-analytic formula for the density of critical values for chi random fields on a general manifold. The result uses Kac-Rice argument and a convenient representation for the Hessian matrix of chi fields, which makes the computation of their expected determinant much more feasible. In the high-threshold limit, the expression for the expected value of critical points becomes very transparent: up to explicit constants, it amounts to Hermite polynomials times a Gaussian density. Our results are also motivated by the analysis of polarization random fields in Cosmology, but they might lead to applications in many different environments. 
\\
	
	\smallskip
	
	\noindent  \emph{AMS  subject classification (2020 MSC)}: 60G60, 
 33C55, 
 53C65, 
58K05  	


\end{abstract}
\maketitle
\tableofcontents 

\section{Introduction}\label{sec:notations}

\subsection{Background}
The investigation of the geometric properties of random fields has represented a major thread of research over the last fifteen years. A major driving force has been given by the publication of very popular research monographs such as \cite{AdlerTaylor} and \cite{AzaisWscheborbook}; these books have discussed in depth the Kac-Rice approach for the derivation of expected values for critical points of smooth random fields. In broad terms, the Kac-Rice approach leads to an "expectation Metatheorem" (the terminology adopted in \cite{AdlerTaylor}), stating that under regularity conditions the expected number of critical points can be expressed in terms of the expectation of the absolute value of the determinant of the Hessian matrix of the field, conditional on the gradient of the field being zero; other conditions can be added to obtain related quantities, for instance on the signature of the Hessian if one is interested in the expected number of minima or maxima. This general approach has led to an impressive amount of results and applications, starting from the celebrated Gaussian Kinematic Formula, which allows the computation of expected value of Lipschitz-Killing Curvatures for excursion sets of Gaussian fields. As noted elsewhere, this area bridges the gap (in a very fascinating way) among different areas of Mathematics, such as Differential Geometry and Random Fields; at the same time, it leads to results which are motivated by fastly growing applied fields, including for instance Cosmology, Neuroimaging, Neural Networks, Optimization, Spin Glasses and many others (see e.g., \cite{Auffinger13},\cite{Cheng17}, \cite{BenArous20}, \cite{Cheng2020}, \cite{Montanari21}, \cite{Fyodorov22}, \cite{Belius22}, \cite{Azais22}, \cite{Telschow23}, just to mention a few recent references).

\subsection{Motivations}
The overwhelming majority of the literature on critical points has so far been confined to the analysis of Gaussian random fields. Indeed, although the Kac-Rice approach is valid in much greater generality than under Gaussianity, it turns out in practice to be extremely hard in non-Gaussian circumstances to derive any analytic expression for critical points: in particular, it is very difficult to compute exactly some extremely cumbersome multiple integrals arising from the absolute values of the Hessian determinants, conditional on the gradient being null. Our purpose in this paper is to move some steps beyond these limitations: more precisely, our goal is to derive some semi-analytic expressions for the density of critical values for chi-square fields defined on the sphere. 

The choice of chi-square fields is natural if one has in mind motivations from statistics or machine learning, and it is easy to figure out several applications. Among these, we are motivated by very concrete examples which arise from Cosmological Data Analysis. In particular, it has been shown in \cite{stec2022GeometrySpin} that chi-square fields may approximate closely the behaviour of the squared norm for \emph{random sections of spin fiber bundles}, i.e. the random fields which model the behaviour of \emph{Cosmic Microwave Background polarisation}, see for instance \cite{MarinucciPeccati11}, Ch.12, \cite{Malyarenko13} or \cite{LiteBird23}. Understanding the distribution of critical points and extrema for polarisation fields is instrumental for the derivation of algorithms allowing point source detection in polarisation data; this would represent an extension of the approach given in the case of scalar random fields (Cosmic Microwave Background temperature data) in \cite{Cheng2020}, see also \cite{Telschow23}, \cite{pistolato2024} and the references therein.

\subsection{Discussion of Main Results}
Our main results can be described as follows. We consider chi random fields fields with $k$ degrees of freedom, defined as the square root of the sum of the squares of $k$ i.i.d., unit variance, \emph{normal} Gaussian fields on a smooth manifold $M$ of dimension $m$. By this we mean the following.
\begin{definition}\label{def:norm}
Let $(M,g)$ be a smooth Riemannian manifold. A \emph{normal field} on $(M,g)$ is a Gaussian field $X$ on $M$, of class at least $\mC^2$, having unit variance: $\E|X(p)|^2=1$ and such that $\E |d_pX(v)|^2=1$ for any unit tangent vector $v\in T_pM$; here and in the sequel, we are using $d_pX$ to denote the differential of $X$ at $p$. A \emph{regular chi--field with $k$ degrees of freedom} on $(M,g)$ is a field $f_k$ of the form 
\be 
f_k=\sqrt{X_1^2+\dots+X_k^2},
\ee
where $X_1,\dots,X_k$ are i.i.d. normal fields. In this case, we say that $f_k$ is \emph{induced by $X_1$}.
\end{definition}
For our applications, we have in mind $M=\mathbb{S}^m$ the unit sphere and $X$ isotropic, i.e., invariant under the action of the orthogonal group, but our results do not require this assumption.
We will show that the expected number of critical points of these fields can be computed, up to some explicit constants, as the expected value of the determinant of random matrices with Gaussian entries. These random matrices do not fall within any known class (such as the Gaussian Orthogonal Ensemble (GOE) or the Gaussian Unitary Ensemble (GUE)), and because these entries have a complicated dependence structure, these expected values in the general case can only be expressed as rather cumbersome multiple integrals, for which it is difficult to provide explicit analytic expressions. More precisely, let us introduce the following:
\begin{definition}\label{def:E(H)}
    Let $H\randin \R^{m\times m}$ be a random symmetric and Gaussian matrix. We say that $H$ is \emph{Hessian-like} if there exists a Gaussian random variable $\gamma\sim N(0,1)$ such that $\E H\gamma=-\mathbbm{1}_m.$
    In this case, we also say that $H$ is \emph{Hessian-like with respect to $\gamma$} and, for all $k\in \N$ and $t\ge 0$, we define the real number
    \be 
E_k^t(H):=\E\kop
1_{[t,+\infty]}(\chi_k)
\left|\det\tyu
A(k-1,m)+\chi_k H+\chi_k(\gamma-\chi_k) \mathbb{1}_m
\uyt\right|\pok\in \R,
    \ee
    where $\chi_k$ is an auxiliary independent chi random variable of parameter $k$ and $A(k-1,m)$ is an auxiliary independent Wishart random matrix, i.e. it is distributed according to \cref{def:A} below.
\end{definition}
Note that if $H$ is Hessian-like and $\gamma$ is as above, then the joint law of $H,\gamma$
is determined by the law of $H$. This explains why we write $E^t_k(H)$ and not $E^t_k(H,\gamma)$. We stress the fact that the only dependence relation among the random variables and matrices in the expectation is that $\E [H\gamma]=-\mathbbm{1}_m.$ 

Our first main result is the following; more discussion on the mentioned random fields and their properties is given in the Sections below:

\begin{thm}\label{thm:one}
Let $(M,g)$ be a smooth $m$-dimensional Riemannian manifold. Let $f_k$ be a regular chi-field with $k<m$ degrees of freedom on $(M,g)$, induced by the normal field $X$.
We have:
\be 
\E[\#C_t]=\frac{\Gamma(\frac{k-m}{2})}{2^{m/2}\Gamma(\frac{k}{2})}\frac{1}{(2\pi)^{\frac m2}}\int_M E^t_k(H_pX)dM(p),
\ee
where $\#C_t$ denotes the number of critical points of values equal or larger than $t$, and $H_pX$ the Hessian matrix of $X$ at $p \in M$.
\end{thm}
Clearly in the special case where the random field $X:M \rightarrow \mathbb{R}$ is isotropic, in an adequate sense, the previous result simplifies to 
\be 
\E[\#C_t]=\frac{\Gamma(\frac{k-m}{2})}{2^{m/2}\Gamma(\frac{k}{2})}\frac{\mathrm{Vol}(M)}{(2\pi)^{\frac m2}}E^t_k(HX),
\ee
where the Hessian $HX$ does not depend on $p$.

The previous result is rather general but suffers from two limitations: the result on the expected value is not fully explicit as the computation of $E^t_k(H_pX)$ requires rather cumbersome multiple integrals (or simulations), and the case $k=m$ is not covered, despite for all $m,k$ the total number of critical points of $\f$ with non zero value is integrable, as we show in \cref{thm:two}.
We are able to address at least partially these issues and obtain our second main result, which we describe below.

More precisely, when we focus on maxima, rather than critical points with an arbitrary signature, we are able to transform the problem into the computation of Gaussian extremes on a different domain, and hence to obtain much more explicit results. In particular, as mentioned above the maxima distribution is strongly motivated by statistical applications, such as the implementation of multiple testing with False Discovery Rate control, as in \cite{Cheng2020}; to this aim, it is especially important to evaluate the distribution of maxima in the high-threshold tail, and especially in the 2-dimensional case $m=2$. Indeed in terms of motivations, it is especially relevant the case where $M=\mathbb{S}^2$ and $k=2$, because as we discussed earlier this corresponds to the modulus of isotropic spin random fields as those emerging from the analysis of Cosmic Microwave Background polarisation data, see also \cite{Carones2024} and the references therein. In this setting, we show that the maxima density takes the form of known polynomials of order $2$ times a Gaussian density. 

In particular, let us denote by $H_{p}(t)$ the Hermite polynomials defined as in \cite[sec. 11.6]{AdlerTaylor}, and by $\phi(t)=(2\pi)^{-1/2}\exp(-t^2/2)$ the standard Gaussian density; we prove the following:
\begin{thm} \label{MainTh} 
Let $X_1,X_2$ be two i.i.d. copies of an isotropic Gaussian field on the two-sphere, of class $\mC^2$, with variance $\E |X_i(p)|^2=1$ and $\E \|d_pX_i\|^2=2 r^2$. Let $f_2(p):=\sqrt{X_1(p)^2+ X_2(p)^2}$ and denote by $\mu_t(\mathbb{S}^2,f_2)$ the number of maxima of $f_2$ where $f_2 \geq t$. Then as $t \rightarrow \infty$ we have that, for some $\delta>0$,
\be 
\tyu 
    \tyu H_2(t)2r^2-2\uyt (2\pi)^{\frac12}\cdot\phi(t)
\uyt^{-1} 
\mathbb{E}[\mu_t(\mathbb{S}^2,f_2)]
=
1 + O \tyu\exp\tyu -\delta t^2\uyt\uyt \ .
\ee 
\end{thm}

In words, the tail behaviour of the maxima distribution is Gaussian, up to corrections terms which are fully explicit combinations of Hermite polynomials and known constants. 

\begin{remark}
In order to be able to connect more easily with the existing literature on isotropic fields (e.g. \cite{MarinucciPeccati11}), we stated the above result for fields that are not normal on the unit sphere, unlike in the rest of the paper. However, the field $X_1$ in \cref{MainTh} becomes a normal field if and only if the sphere is endowed with the round metric of radius $r$, in which case $f_2$ becomes a regular chi-field with $2$ degrees of freedom, induced by $X_1$.
\end{remark}
Theorem \ref{MainTh} can be seen a corollary of a more general result that is valid in full generality and shows that the behavior of the expected number of maxima of a regular chi field resembles in some aspects that of a Gaussian one, which is well documented in the literature. In particular, by the aforementioned passage from $f_k$ to an auxiliary Gaussian field $\f$, we are able to exploit the results of  \cite{GayetExEulMax} and \cite{AdlerTaylor}, proving the following.
\begin{thm}\label{LastMainTh}
Let $f_k$ be a regular chi-field with $k$ degrees of freedom (see \cref{def:norm}) on a smooth compact Riemannian manifold $M$ of dimension $m$. For any Borel subset $A\subset M$, let us denote by $\mu_t(A,f_k)$ the number of maxima of $f_k$ where $f_k \geq t$ that belong to $A$. If $f_k$ is induced by the normal field $X$, then we have that
    \be
\E\tyu  \mu_t(A,f_k) \uyt
=
\frac{1}{2^{\frac{m+k-3}{2}}\pi^{\frac{m-1}{2}}
\Gamma\tyu \frac{k}{2}\uyt}
\int_A D^t_k\tyu[H_pX]\uyt dM(p),
    \ee
where $D_k^t([H_pX])$ depends solely the law of $H_pX$ and is defined in \cref{def:D}.
Moreover, as $t\to +\infty$, the following asymptotic equivalences hold up to an error of order $O \tyu\exp\tyu -(\frac12+\delta) t^2\uyt\uyt$ for some $\delta>0$:
\begin{gather}
\P \tyu \max_{M} f_k \ge t \uyt \sim \E\tyu  \mu_t(M,f_k) \uyt
\\
\sim 
\E(\#\{df_k=0, f_k \ge t\}) 
\sim \E b(f_k \ge t)
\sim\E b_0(f_k\ge t) 
\sim \E b_0(f_k\ge t; \mathbbm{B}^m)
\\
\sim \E\chi(f_k\ge t)
\sim \sum_{j=0}^{m+k-1}\frac{\m L_{j}(M\times \mathbb \S^{k-1})}{(2\pi)^{\frac{j}2}}H_{j-1}(t)\phi(t)
\\
\overset{\text{(if $M=r\mathbb \S^2$ and $k=2$})}{\sim} (2+2r^2H_2(t))\sqrt{2\pi} \phi(t) .
\end{gather}
Here, $b(f_k\ge t)$ is the sum of all Betti numbers; $b_0(f_k\ge t)$ is the number of connected components; $b_0(f_k\ge t, \mathbb B^m)$ is the number of the connected components that are homeomorphic to the unit ball in $\R^m$; $\chi(f_k\ge t)$ is the Euler--Poincaré characteristic; $\m L_i$ is the $i^{th}$ Lipschitz--Killing curvature (defined as in \cite[sec. 7.6]{AdlerTaylor}).
\end{thm}
Note in particular that the asymptotic behavior of the excursion probability of $f_k$ at a high threshold depends only on the geometry of $M$ and not on the inducing normal field $X$ (which is not uniquely determined by the Riemannian metric of $M$), although the distribution of $\max_{M} f_k$ might depend on $X$. Moreover, we can observe that 
\be 
\P \tyu \max_{M} f_k \ge t \uyt 
\sim 
\P \tyu \max_{M\times \mathbb \S^{k-1}} \f \ge t \uyt,
\ee
for any normal Gaussian field $\f$ defined on $M\times \mathbb \S^{k-1}$, in virtue of \cite[Th. 12.4.1]{AdlerTaylor}. Indeed, the main idea of our proof will be to show that for a suitable normal field $\f$, we have that $\mu_t(M,f_k)=\mu_t(M\times \mathbb \S^{k-1},\f )$, see \cref{sec:asymp} (see also \cite{Kuriki_Matsubara_2023} and \cite{Bloomfield:2016civ} for some related results on the geometry of chi-square fields with a view to cosmological applications).
\subsection{Plan of the paper}
The plan of the paper is as follows: in Section \ref{sec:Setting} we fix our notation and introduce some background material; in Section \ref{sec:critical values} we give our general result for critical values, which is not fully explicit: for this reason, in Section \ref{sec:Hessian} we study more deeply the structure of the Hessian in two dimension and in Section \ref{sec:asymp} we exploit these results to give a fully analytic expression for the expected value of the number of maxima, and in \cref{sec:HighTresh} we prove the high-threshold limits. We first prove \cref{MainTh} directly, then prove the more general \cref{LastMainTh} by relying on more abstract results.

\subsection{Acknowledgements}
The research leading to this paper has been supported by PRIN project \emph{Grafia} (CUP: E53D23005530006), PRIN Department of Excellence \emph{MatMod@Tov} (CUP: E83C23000330006) and by the Luxembourg National Research Fund (Grant: 021/16236290/HDS). DM is a member of Indam/GNAMPA.

\section{Setting and Background} \label{sec:Setting}
\subsection{Notations}
The following list contains some recurring conventions adopted in the rest of paper.
\begin{enumerate}[(i)]
\item {Unless otherwise specified, every random element is assumed to be defined on an adequate common probability space $\Prob$.}
\item A \emph{random element} (see \cite{Billingsley}) of the topological space $V$ (or \emph{with values} in $V$) is a measurable mapping $X\colon \Omega\to V$, defined on $\Prob$. In this case, one writes
    \be\label{eq:randin}
    X\randin V
    \ee 
    and denote by $[X]:=\PP X^{-1}$ the (push-forward) Borel probability measure on $V$ induced by $X$. We will use the notation
\be 
\PP\{X\in U\}:=[X](U)=
\PP X^{-1}(U)
\ee 
to indicate the probability that $X\in U$, for some Borel measurable subset $U\subset V$, and write (as usual)
\be 
\E\{f(X)\}:=\int_{V}f(v)[X](dv),
\ee
to denote the expectation of the random variable $f(X)$, where $f\colon V\to \R^k$ is a measurable mapping such that the above integral is well-defined.
We will sometimes write that $X$ is a \emph{random variable}, a \emph{random vector} or a \emph{random field}, respectively, when $V$ is the real line, a vector space, or a space of functions $\mC^r(M,\R^k)$, 
respectively. 
\item We will use the special symbol
\be 
X\colon M \randto \R^k,
\ee 
to indicate that $X$ is a random field (see above), i.e., a random element of $\mC^0(M,\R^k)$. The symbol hints at the fact that $X$ is also a measurable function $X\colon M\times \Omega\to \R^k$.
\item The sentence: ``$X$ has the property $\mathcal{P}$ almost surely'' (abbreviated ``a.s.'') means that the set  $S=\{v\in V : v \text{ has the property }\mathcal{P}\}$ contains a Borel set of $[X]$-measure $1$. It follows, in particular, that the set $S$ is $[X]$-measurable, i.e. it belongs to the $\sigma$-algebra obtained from the completion of the measure space $(V,\mathcal{B}(V),[X])$.
 \item We write $\#(S)$ for the cardinality of the set $S$.

\end{enumerate}
\subsection{Definition of the main objects}
\subsubsection{Normal fields}\label{sec:normal}
Let $(M,g)$ be a smooth manifold of dimension $m$ and let $X\colon M\randto \R$ 
be a Gaussian random field of class $\mC^2$ such that for all $p\in M$ we have $X(p)\sim \mathcal{N}(0,1)$ and 
\be 
g_p(v,w)=\E\kop d_pX(v)d_pX(w)\pok.
\ee
We call $g_p(v,w)$ the \emph{Adler and Taylor metric}, see \cite[Section 12.2]{AdlerTaylor}. Following \cite[Definition 6.3]{MathiStec}, in this case we write $X\sim \mathcal{N}(M,g)$ and say that $X$ is a \emph{normal field} on $(M,g)$, as anticipated in \cref{def:norm}. 
Recall that the Hessian is the random bilinear form $H_pX\colon T_pM\times T_pM\to \R$ such that
\be 
H_pX(v,w)=\de_v\de_w X(p)-d_pX(\nabla_vw),
\ee
where $\nabla$ is the Levi-Civita connection of $g.$
Let us first recall the following standard characterization of the dependence structure for the gradient and Hessian. 
\begin{prop}\label{prop:covjet}
\begin{enumerate}
For every $p\in M$,
\item $X(p)$ and $d_pX$ are independent.
    \item $d_pX$ and $H_pX$ are independent.
    \item $\E\{X(p)H_pX\}=-g_p$.
\end{enumerate}
\end{prop}
\begin{proof} These results are classical and they are proved, for instance, in 
\cite[Section 12.2]{AdlerTaylor}.
\end{proof}
\subsubsection{Chi distribution}
\begin{definition}
Let $k\in\N.$ We say that a random variable $\a\randin \R$ is a \emph{chi of parameter $k$} if it has the same law as the random variable 
\be 
\chi_k:=\sqrt{\gamma_1^2+\dots +\gamma_k^2},
\ee
    where $\gamma_1,\dots,\gamma_k\sim N(0,1)$ are indipendent and identically distributed. In this case, we will write briefly that $\a\sim \chi_k.$ The following characterization is classical, but we recall it for completeness.
\end{definition}
\begin{prop}
    Given $a\in \R,$ we have that $\chi_k\in 
    L^{a}$ if and only if $k>-a.$
\end{prop}
\begin{proof}
  It is sufficient to observe that
\be 
\E\{\chi_k^{a}\}=\int_{\R^k} |x|^a \frac{e^{-\frac{|x|^2}{2}}}{(2\pi)^{\frac{k}{2}}} d\R^k(x)\sim \int_{0}^1r^{a+k-1}dr.
\ee
\end{proof}

Before we state our first main result, let us recall a simple property of chi-random variables; by a straightforward computation, for $k>m$ we have that
\be
\E\kop\frac{1}{\chi_k^m}\pok =\int_0^{\infty} \frac{1}{x^{m/2}}\frac{x^{k/2-1}exp(-x/2)}{\Gamma(k/2)2^{k/2}} dx 
\ee
\be
=\frac{\Gamma(\frac{k-m}{2})}{2^{m/2}\Gamma(\frac{k}{2})}\int_0^{\infty} \frac{x^{(k-m)/2}e^{-x/2}}{\Gamma(k-m)/2)} dx =\frac{\Gamma(\frac{k-m}{2})}{2^{m/2}\Gamma(\frac{k}{2})} \ .
\ee
\subsubsection{The chi-field}
Now let $Y:=(X^1,\dots,X^k):M\randto \R^k$ such that all components are i.i.d., $X^i\sim X.$ 
Define $F:M\randto \R$ as 
\be 
F(p)=\frac12 |Y(p)|^2 \ ;
\ee
in particular, notice that $F(p)\sim \frac12\chi_k^2$ for all $p\in M.$ Denote $Z:=F^{-1}(0)$ and for $t\ge 0$,
\bega 
C_t:&=\text{Crit}(|Y|)\cap \kop |Y|\ge t\pok
\\
&=\text{Crit}(F)\cap \{F\ge \frac{t^2}{2}\}=\kop p\in M: d_pF=0, F(p)\ge \frac{t^2}{2} \pok.
\eega
Of course, $Z$ denotes the nodal set of $Y$ while $C_t$ counts the number of critical values where the chi-field is larger than some given (positive) value $t$. Note that $Z\subset C_0\subset M$ is a random submanifold of dimension $d=m-k$ and $C_t\subset C_0\smallsetminus Z$ is a random finite set for all $t>0.$
Our first goal is to compute the expected value $\E[\#C_t]$; indeed, in the language of \cref{sec:notations}, the field $f=|Y|$ is a regular chi-field with $k$ degrees of freedom, on the manifold $M$.
\begin{remark}
    The Riemannian volume density of $(M,g)$, which we denote as $dM$, is proportional to the expectation of the Riemannian $d$-volume of $Z=Y^{-1}(0):$
\be 
\frac{1}{s_{d}}\E\kop \int_Z f(p)dZ(p)\pok=\frac{1}{s_m}\int_M f(p)dM(p),
\ee
where $s_i$ is the $i$-dimensional volume of the unit sphere of dimension $i$: $\S^i\subset \R^{i+1}$. This expression is a consequence of Kac-Rice formula \cite[Theorem 6.8]{AzaisWscheborbook}. The precise constants can be computed by testing the formula on spheres, see \cite[Proposition 95]{stec2022GeometrySpin}.
\end{remark}
\subsubsection{Random matrices}
\begin{definition}\label{def:A}
Let $k,m\in \N$. Let $\gamma_1,\dots,\gamma_m\sim N(0,\mathbbm{1}_k)$. We define $A(k,m)\randin \R^{m\times m}$ to be the random symmetric matrix whose coordinates $A_{a,b}$ have a joint law defined by:
\be 
A_{a,b}=\langle\gamma_a,\gamma_b\rangle.
\ee
\end{definition}

Notice that $A\sim R^TAR$ for any orthogonal matrix $R\in O(m).$ For instance,
\be 
A(1,2)=\begin{pmatrix}
\gamma_1^2 & \gamma_1\gamma_2 
\\ \gamma_1\gamma_2  & \gamma_2^2
\end{pmatrix} ,
\ee
\be 
A(2,2)=\begin{pmatrix}
\gamma_{11}^2 +\gamma_{12}^2 & \gamma_{11}\gamma_{21}+\gamma_{12}\gamma_{22} 
\\ \gamma_{11}\gamma_{21}+\gamma_{12}\gamma_{22}  & \gamma_{21}^2 +\gamma_{22}^2
\end{pmatrix} ,
\ee
\be 
A(3,2)=\begin{pmatrix}
\gamma_{11}^2 +\gamma_{12}^2+\gamma_{13}^2 & \gamma_{11}\gamma_{21}+\gamma_{12}\gamma_{22}+\gamma_{13}\gamma_{23} 
\\ \gamma_{11}\gamma_{21}+\gamma_{12}\gamma_{22}+\gamma_{13}\gamma_{23}  & \gamma_{21}^2 +\gamma_{22}^2+\gamma_{23}^2
\end{pmatrix} .
\ee

\begin{remark}
Matrices of the form $A(k,m)$ follow a so-called Wishart distribution  $A(k,m) \sim W_m(\mathbbm{1}_m,k)$; more precisely, for $k \geq m$ these matrices have densities
\be
f_{(k,m)}(A)=\frac{(\det(A))^{(k-m-1)/2}\exp(-\tr(A/2))}{2^{km/2}\pi^{m(m-1)/4}\prod_{j=1}^m\Gamma((k+1-j)/2)}\mathbbm{1}_{\det(A)>0}(A) .
\ee
It can be noted that the law of the matrix $A$ depends just on its determinant and its trace - two quantities invariant to rotations, as expected; moreover, these densities are positive only over matrices which are positive definite, and they are zero otherwise. For instance, we have
\be
f_{(2,2)}(A)=f_{(2,2)}(a_{11},a_{12},a_{21},a_{22})=\frac{\exp((-a_{11}-a_{22})/2)}{4\pi(a_{11}a_{22}-a_{12}a_{21})^{1/2}}\mathbb{I}_{\det(A)>0}(A) ,
\ee
and
\be
f_{(3,2)}(A)=f_{(2,2)}(a_{11},a_{12},a_{21},a_{22})=\frac{\exp((-a_{11}-a_{22})/2)}{4\pi}\mathbb{I}_{\det(A)>0}(A) .
\ee
\end{remark}





Now recall the notion of Hessian-like matrices in Definition \ref{def:E(H)}, and notice that $E^t_k(H)=E^t_k(R^THR)$ for any orthogonal matrix $R\in O(m).$ Moreover, the property of being Hessian-like is also invariant under orthogonal changes of coordinates. Therefore, the following definition is well posed.
\begin{definition}\label{def:H}
    Let $(T,g)$ be any Euclidean space of dimension $m$ and let $k\in \N$ and $t>0.$ Let $H\colon T\times T\to \R$ be a Hessian-like Gaussian symmetric bilinear form on $T$. Then we define the deterministic real number
    \be 
E^t_k(H):=E^t_k\tyu (H(e_a,e_b))_{1\le a,b\le m}\uyt\in \R,
    \ee
    where $e_1,\dots ,e_m$ is any orthonormal basis of $T.$ To keep track the dependence on the metric $g$, when needed, we will write $E^t_k(g,H)$. 
\end{definition}
\begin{lemma}
    $E^t_k(\lambda g,H)=E^t_k(g,\lambda^{-1}H)$ for any $\lambda>0$.
    \end{lemma}
    \begin{proof}
The proof is straightforward and hence omitted.
    \end{proof}
\begin{remark}
$E^t_k(H)$ depends only on covariance matrix of $H$, that is, it depends on the tensor $\E H_{ab}H_{cd}$. And for any $R\in O(m),$
\be 
E^t_k\tyu \kop\E H_{ab}H_{cd} :abcd\pok \uyt=E^t_k\tyu \kop\E (R^THR)_{ab}(R^THR)_{cd} :abcd\pok\uyt.
\ee
\end{remark}
\section{The First Main Result: Critical Points}\label{sec:critical values}

In this Section we give our main result on the expected value of critical points for chi fields. For convenience, we split it into two subsections, when covering the case $k>m$, the other $k=m$ which requires some different argument.

\subsection {The expected value of critical points for $k>m$} 
Here is the main result of this subsection.

\begin{thm}\label{thm:onetwo}
In the setting described above, for all $k>m$, we have:
\be 
\E[\#C_t]=\frac{\Gamma(\frac{k-m}{2})}{2^{m/2}\Gamma(\frac{k}{2})}\frac{1}{(2\pi)^{\frac m2}}\int_M E^t_k(H_pX)dM(p).
\ee
\end{thm}
\begin{proof}
Notice that when $k>m$, the set $Z=Y^{-1}(0)$ is almost surely empty. Let us observe that
\be 
d_pF(v)=Y(p)^Td_pY(v)\randin \R;
\ee
\be 
H_pF(v,w)=d_pY(v)^Td_pY(w)+Y(p)^TH_pY(v,w)\randin \R.
\ee
Notice that, by \cref{prop:covjet}, we have that $d_pY$ and $Y(p)$ are independent, for every fixed $p\in M.$

We will use the Kac-Rice formula (in particular, we refer to the statement \cite[Alpha-formula]{MathiStec}). Assuming that the formula is applicable, we have that
\bega
\E\kop\#C_t\pok&=\E\kop \sum_{p\in dF^{-1}(0)} 1_{[\frac{t^2}{2},+\infty]}(F(p))\pok
\\
&=\int_M\E\kop 1_{[\frac{t^2}{2},+\infty]}(F(p))|\det (H_pF)|\Big| d_pF=0\pok\rho_{d_pF}(0)dM(p),
\eega
where, for any $p\in M$ fixed, $\rho_{[d_pF]}\colon T_p^*M\to [0,+\infty)$ is the density of the random vector $d_pF\randin T_p^*M$, with respect to the volume defined by the (flat) metric $g_p$.
Indeed, $x\mapsto \rho_{[d_pF]}(x)$ exists if and only if $k>m$. In this case, it is continuous with respect to both $(p,x)$. Let us compute it. Let us fix an orthonormal basis of $T_p^*M$, so that $(T_p^*M,g_p)\cong (\R^m,\mathbbm{1}_m).$ Then for all bounded continuous functions $\a\colon \R^m\to [0,1]$ we have:
\bega 
\int_{\R^m}\a(x)\rho_{[d_pF]}(x)d\R^m(x)&=\E\kop \a(d_pF) \pok
\\
&=\int_{\R^k}\E\kop \a(u^Td_pY) \big|Y(p)=u\pok d[Y(p)](u)
\\
&=\int_{\R^k}\E\kop \a(u^Td_pY) \pok d[Y(p)](u)=\dots
\eega 
where we used the expression $[Y(p)]$ to denote the probability measure induced by $Y(p)\in\R^k$, i.e., the law of $Y(p).$ Now observe that, by construction, the law of the random matrix of $d_pY$ in an orthonormal basis is that of the $k\times m$ matrix:
\be\label{eq:distrY} 
d_pY=\tyu\gamma_{j}^{i}\uyt_{1\le i\le k,1\le j\le m},
\ee
where $\gamma_j^{i}\sim N(0,1)$ are i.i.d. This distribution is invariant under orthogonal transformations, therefore, the integrand above depends only on $|u|.$ Observe that the law of $|Y(p)|$ is that of a chi of parameter $k$, that we have denoted as $\chi_k.$ Hence we obtain
\bega 
\dots&=\int_{0}^{+\infty}\E\kop \a(t(e_1)^Td_pY) \pok d[|Y(p)|](t)
\\
&=\E\kop \a(\chi_k\cdot d_pX^1) \pok
=\E\kop \int_{\R^m} \a(\chi_k\cdot x) \frac{e^{-\frac{|x|^2}{2}}}{(2\pi)^{\frac m2}}d\R^m(x)\pok
\\
&=\E\kop \int_{\R^m} \a(y) \frac{e^{-\frac{|y|^2}{2\chi_k^2}}}{(2\pi)^{\frac m2}\chi_k^m}d\R^m(y)\pok
=\int_{\R^m} \a(x) \E\kop \frac{e^{-\frac{|x|^2}{2\chi_k^2}}}{(2\pi)^{\frac m2}\chi_k^m}\pok d\R^m(x) \ .
\eega
In the 4th identity above, we used the change of variables $y=\chi_k \cdot x.$ We conclude that
\be 
\rho_{[d_pF]}(x)=\E\kop \frac{e^{-\frac{|x|^2}{2\chi_k^2}}}{(2\pi)^{\frac m2}\chi_k^m}\pok, \text{ for almost every $x\in \R^m.$}
\ee
If $m<k$, this defines a continuous function of $(p,x)$ (if $m\ge k$, it has a pole at $x=0$) and $\rho_{[d_pF]}(0)=\E\kop (2\pi)^{-\frac m2}\chi_k^{-m}\pok$.

Now, let us compute the conditional probability given $d_pF$. This is interpreted as a family of random vectors parametrized by the possible values of $d_pF$ and we denote it as $[(F(p),H_pF)|d_pF=\xi],$ for $\xi\in T_p^*M$. We are only interested in the case $\xi=0.$ We will do the computation in an orthonormal frame, so that $T_p^*M=\R^m$. Let $\a\colon \R^{m\times m}\to [0,1]$ be any continuous function. Then
\bega 
\E&\kop\a(F(p),H_pF)\ |d_pF=0\pok=\int_{\R^k}\E\kop\a(F(p),H_pF)|d_pF=0, Y(p)=u\pok d[Y(p)](u)
\\
&=\int_{\R^k}\E\kop\a\tyu\frac{|u|^2}{2},H_pF\uyt\Bigg|\begin{aligned}d_pF&=0, \\ Y(p)&=u\end{aligned}\pok d[Y(p)](u)=
\\
&=\int_{\R^k}\E\kop\a\tyu \frac{|u|^2}{2},d_pY^Td_pY+u^TH_pY\uyt \Bigg|\begin{aligned}u^Td_pY&=0, \\ Y(p)&=u\end{aligned}\pok d[Y(p)](u)=\dots
\eega
Recall that the law of $d_pY$ is that given in \eqref{eq:distrY} and that it is independent from $(H_pY,Y(p))$. Moreover, by \cref{prop:covjet}, we have that $[H_pX|X(p)=t]=[H_pX+(X(p)-t)g_p]$. Finally, as before, the integrand depends only on $|u|$ (indeed that $Y\sim RY$ for any $R\in O(k)$), so that we can continue as follows.
\bega 
\dots&=\int_{0}^{+\infty}\E\kop\a\tyu \frac{t^2}{2},\sum_{i=2}^k (d_pX^i)^Td_pX^i+tH_pX^1\uyt\Bigg|\begin{aligned}d_pX^1&=0, X^1(p)=t,\\ X^i(p)&=0\ \forall i\ge 2\end{aligned}\pok d[|Y(p)|](t)
\\
&=
\E\kop\a\tyu\frac{\chi_k^2}{2},\sum_{i=2}^k (d_pX^i)^Td_pX^i+\chi_k(H_pX+(X(p)-\chi_k)g_p)\uyt\pok =\dots
\eega
Notice that the random matrix $A:=\sum_{i=2}^k (d_pX^i)^Td_pX^i$ has coordinates $A_{a,b}= \sum_{i=2}^k \gamma_a^i\gamma_b^i$, therefore ${A\sim A(k-1,m)},$ as in  \cref{def:A}.

Moreover, $H_pX$ is obviously a Hessian-like Gaussain matrix, in the sense of \cref{def:H} and $\gamma:=X(p)\sim N(0,1)$ is, by \cref{prop:covjet}, the associated Gaussian random variable such that $\E\{H_pX\cdot X(p)\}=-g_p=-\mathbbm{1}_{m}.$ Since the above identities are true for arbitrary $\a$, we can interpret them as identities of probability laws, to conclude that:
\bega 
\qwe\tyu F(p),H_pF\uyt |d_pF=0\ewq=
\qwe \tyu \frac{\chi_k^2}2,A(k-1,m)+\chi_k\cdot H_pX+\chi_k(\gamma-\chi_k)\mathbbm{1}_{m})\uyt\ewq,
\eega
where the only dependence relation is $\E H_pX\gamma=-g_p.$

Now that we have all the ingredients, for $k>m$, we can write the Kac-Rice formula and conclude:
\bega
\E\kop\#C_t\pok&=\int_M\E\kop 1_{[\frac{t^2}{2},+\infty]}(F(p))|\det (H_pF)|\Big| d_pF=0\pok\rho_{d_pF}(0)dM(p)
\\
&=\int_M \E\kop 1_{[t,+\infty]}(\chi_k)|\det\tyu A(k-1,m)+\chi_k\cdot H_pX+\chi_k(\gamma-\chi_k)\mathbbm{1}_{m}\uyt |\pok \cdot 
\\
&\quad \cdot \E\kop\frac{1}{(2\pi)^{\frac{m}{2}}\chi_k^m}\pok dM(p)
\\
&=\int_M E^t_k(H_pX)  \E\kop\frac{1}{(2\pi)^{\frac{m}{2}}\chi_k^m}\pok dM(p).
\eega
\end{proof}

\subsection{The general case: including $m= k$}

The fact that the theorem holds only for $m<k$ seems to be due to a strange phenomenon of Kac-Rice formula: sometimes the Kac-Rice density, written as ``conditional expectation times density'', contains some expression of the form: $0\cdot \infty.$ In this case, it might be that $0\cdot \infty\in \R$. Indeed, in principle, it is possible that there exists another function $F_\e$ with the same high level critical points as $F$, for which Kac-Rice formula can be applied. Indeed, when $k=1$, the critical points of $F$ of level $t>0$ are exactly the critical points of the normal field $X$, of level $\pm t$, and such expectation can be computed with standard computations.

We have failed trying to find a good modification of $F$. However, 
we prove below that $\E\{\#C_t\}$ is finite whenever $C_t$ is almost surely a finite set.

There is indeed a generalized version of Kac-Rice formula that can be applied directly to our situation, as we explain below.\\
Let $\R^k_M=M\times \R^k$ denote the trivial vector bundle over $M$ of rank $k$. Let us consider the space of one jets of $k$-valued functions:
\be 
J^1(M,\R^k):=\R^k\times \tyu T^*M\uyt^{\oplus k}
=\kop (p,y,A): p\in M, y\in \R^k, A:T_pM\to \R^k \text{ linear}\pok,
\ee
where a linear map $A:T_pM\to \R^k$ is seen as a $k$-tuple of covectors $A^1,\dots,A^k\in T^*_pM,$ which are its ``rows''. 

Recall (see \cite{Hirsch}) that to any smooth function $Y\randin \mC^1(M,\R^k),$ we can associate a smooth \emph{1-jet prolongation} $j^1Y:M\to J^1(M,\R^k)$ defined as 
\be 
j^1Y(p):=j^1_pY:=(p,Y(p),d_pY).
\ee
Clearly, $P:J^1(M,\R^k)\to M$ is a smooth vector bundle and $j^1Y$ is a smooth section. We will use the following standard notation for the fiber of this vector bundle: for any $p\in M$
\be 
P^{-1}(p)=:J^1_p(M,\R^k).
\ee

For any $t\ge 0,$ define the subset $W_t\subset J^1(M,\R^k)$ such that
\be 
W_t:=\kop (p,y,A) \in J^1(M,\R^k): |y|> t \text{ and }y^TA=0\pok.
\ee
The closure of $W_t$ is just the set $\overline{W_t}=W_t\cup \de W_t$, where 
\be 
\de W_t:=\kop (p,y,A) \in J^1(M,\R^k): |y|= t \text{ and }y^TA=0\pok.
\ee
Now, observe that $W_t$ has codimension $m$ and that the set that we are studying is
\be 
C_t=(j^1Y)^{-1}(\overline{W_t}).
\ee

It is easy to see that $W_t\subset J^1(M,\R^k)$ is an open semialgebraic (locally, because it is defined by polynomial inequalities) submanifold for all $t\ge 0$. Indeed, in a local chart defined on an open subset $O \subset M$ we have that
\bega\label{eq:local}
W_t\cap P^{-1}(O) \cong \kop (p,ry,(A_1,\dots ,A_m))\in O\times \R^k\times \R^{k\times m}: p\in O, r> t, y\in \S^{k-1}, A_i\in y^\perp \pok,
\eega 
where here $A_1,\dots,A_m$ are the columns of $A:\R^m\to \R^k.$ Thus, it follows that $W_t$ is locally diffeomorphic to 
\be 
W_t\cap P^{-1}(O)\cong O\times \tyu (t,+\infty)\times \tyu T\S^{k-1}\uyt^{\oplus m}\uyt.
\ee

If $t>0,$ the closure $\overline{W_t}=W_t\cup \de W_t$ is a manifold with boundary.
In the case $t=0,$ the topological frontier $\de W_t$ is not a smooth boundary, but rather an additional stratum of codimension $k$:
\bega\label{eq:local}
\overline{W_0} \cap P^{-1}(O)\cong \tyu O\times \{0\}\times \R^{k\times m} \uyt\cup W_0 \cap P^{-1}(O).
\eega 
This stratum is semialgebraic thus, the union $\overline{W_0}$ remains a semialgebraic subset with two strata. Its codimension is, by definition, the minimimum of the codimensions of the two strata. Therefore $\overline{W_0}$ has codimension $m$ if and only if $m\le k$ and $\overline{W_t}$, for $t>0,$ has always codimension $m.$

\begin{remark} 
Observe that when $k> m,$ the stratum $\de W_0$ is too small and $j^1Y^{-1}(\de W_0)=\{Y=0\}$ is almost surely empty. In the general case, $j^1Y^{-1}(\de W_0)=\{Y=0\}\subset M$ is almost surely a submanifold of dimension $m-k.$ In particular, we will certainly have $\E\#C_0=\infty$ if $k<m.$
\end{remark}

We are in the position to apply \cite[Thm. 27]{KRStec} to deduce the following. Notice that the case $k=m$ was not included in \cref{thm:one}.
\begin{thm}\label{thm:two}
For all $k\in \N$ and $t>0$, we have $\E\{\#C_t\}\le\E \{\# C_{0^+}\}$, where
\be 
\E \{\# C_{0^+}\}:=\E\{\#\cup_{t>0}C_t\}<+\infty.
\ee 
Moreover, $\E\{\#Z\}=\E\{\#C_0\}-\E \{\# C_{0^+}\}$ is the expected number of zeroes of $Y$ and satisfies the following:
if $k> m,$ then $\E\{\#Z\}=0$; if $k=m,$ then $\E\{\# Z\}\in (0,+\infty)$; if $k< m,$ then $\E\{\# Z\}=+\infty$.
\end{thm}
\begin{proof}
Following the discussion above, we have to show that $\E\{\#(j^1Y)^{-1}(W_0)\}$ is finite. Let $W=W_0$ and let $\pi\colon E\to M$ be the trivial vector bundle $E:=M\times \R^k$. 
We will apply \cite[Cor. 3.9]{KRStec} to the random field $Y\colon M\randto \R^k$, that in the language of \cite[Cor. 3.9]{KRStec}, is a smooth Gaussian random section of $E$. The fiber over $p\in M$ of its $1$-jet extension is $J^1_pE=J^1_p(M,\R^k)=P^{-1}(p).$

We already observed that $W\subset E$ is a semialgebraic submanifold of codimension $m.$ by \cite[Rem. 3.3]{KRStec}, this implies that $W$ has sub-Gaussian concentration.
The fact that $W$ is transverse to the fibers $P^{-1}(p)$ for all $p\in M$ is obvious from \cref{eq:local}, in that the local equations of $W$ do not involve $p$. 

The $1$-jet of $Y$ at $p\in M$ is
    \be 
j^1_pY=(p,Y(p),d_pY)\randin J^1_p(M,\R^k),
    \ee
    which is non-degenerate by construction since its support is the whole space $\{p\}\times \R^k\times (T_p^*M)^k=J^1_p(M,\R^k)$.

We checked all hypotheses for point 1. of \cite[Cor 3.9]{KRStec}, applied to the field $Y$, which implies that $\E\{\#(j^1Y)^{-1}(W)\}$ is finite given that the manifold is compact.

Regarding the set of zeroes $Z=Y^{-1}(0)$, we have that if $k>m$, then $Z$ is almost surely empty while if $k\le m$, we can use \cite[Theorem 6.2]{MathiStec} using the same argument as in the proof of \cite[Lemma 6.5]{MathiStec} to compute 
\be\label{eq:vols} 
\E\kop \vol_{m-k}(Z) \pok=\frac{\vol_{m-k}(\S^{m-k})}{\vol_m(\S^m)}\vol_m(M),
\ee
where $\vol_j$ denotes the $j^{th}$ Hausdorff volume measure associated to the Riemannian manifold $(M,g)$. Clearly, \eqref{eq:vols} implies the thesis.
\end{proof}

\begin{remark}
The condition $k\ge m$ in the above theorem is due to the fact that $C_0$ includes all critical values $v$ satisfying $v\ge 0$. For instance, when $k=1$ and $Y=X$, we have that $C_0=\{X=0\}\cup \{dX=0\}$ is clearly infinite, whilst $\E C_{0^+}=\E\{dX=0\}$ is finite.
\end{remark}
\section{A study of the Hessian in dimension $2$}\label{sec:Hessian}
Consider the case in which the original Gaussian random function $X:M\randto \R$ is a random eigenfunction on the sphere $M=\S^m.$ Then, $X$ is invariant in law under isometries of $\S^m.$ In particular, this implies that for any $R\in O(m+1),$ that fixes a point $p\in \S^m$ we have that 
\be 
R^T H_pX R=H_p(X\circ R)\sim H_pX.
\ee
the same happens for stationary fields on $\R^d$, like Berry or Bargmann-Fock.
We may regard such isometry $R$ as an isometry of $p^\perp=T_p\S^m$ and generalize this concept. 
\begin{definition}\label{def:rothessian}
Given a Gaussian bilinear form $H$ on a Euclidean space $V$, we say that $H$ is \emph{rotation invariant} if for every $R\colon V\to V$ linear orthogonal transformation, there is an equivalence in law:
\be\label{eq:rotinv} 
H\sim R^TH R.
\ee
Given a Gaussian field $X\colon M\randto \R$, we say that $X$ has \emph{rotation invariant Hessian} if for every point $p\in M$ the Hessian $H_pX$ is a rotation invariant, as a random bilinear form on $T_pM$.
\end{definition}
\begin{lemma}\label{lem:rothessional}
Let $H=\begin{pmatrix}
    h_1 & h_{3} \\ h_{3} & h_2
\end{pmatrix}$ satisfy \cref{eq:rotinv}. Then there are constants $\sigma,c \ge 0$ such that:
\be\label{eq:rothessian}
\begin{pmatrix}
h_1\\
h_2\\
h_{3}
\end{pmatrix}
\sim \mathcal{N} \left ( 0, \begin{pmatrix}
2\sigma^2+c  &    c            &  0    \\
c            & 2\sigma^2 +c  &  0    \\
0            &   0             & \sigma^2   
\end{pmatrix}          \right ) .
\ee
Moreover, $H$ is Hessian-like if and only if $(\sigma^2+c)\ge 1$, with 
\be
\gamma=-\tr(H)\frac{1}{2(\sigma^2+c)}+\gamma_0 \sqrt{1-\frac{1}{(\sigma^2+c)}}.
\ee
for some $\gamma_0\sim \m N(0,1)$ independent from $H$.
\end{lemma}
\begin{proof}
First, one can easily see that $\E[h_1^2]=\E[h_2^2]=:a^2$, and that $\E[h_1h_3]=\E[h_2h_3]=:b$. Let $R(\theta)$ be the matrix of the rotation of angle $\theta$ in $\R^2$ and define
\be 
H(\theta)=\begin{pmatrix}
    h_1(\theta) & h_{3}(\theta) \\ h_{3}(\theta) & h_2(\theta)
\end{pmatrix}:=R(\theta)^THR(\theta).
\ee
\end{proof}
By imposing the condition that $\E[h_i(\theta)h_j(\theta)]$ is constant in $\theta$, one gets, for any choice of $i,j$, the same condtion:
\be 
a^2=2\sigma^2+c\quad \text{and}\quad b=0;
\ee
hence the proposition is proven.
\subsubsection{Stationary plane fields}
\begin{prop}\label{prop:stathess}
Let $\xi\colon \R^d\randto \R$ be a stationary and isotropic Gaussian random field of class $\mC^2$, with covariance function $K(|x-y|)=\E\kop \xi(x)\xi(y)\pok.$ Then the random variables $\de_i\de_j\xi(0)$, for $1\le i\neq j \le d$ have the following covariances:
\be 
\E\de_{i,j}\xi(0)\de_{i,k}\xi(0)=\E\de_{i,j}\xi(0)\de_{h,k}\xi(0)=0,
\ee
\be 
\E|\de_i^2\xi(0)|^2=K''''(0), \quad \E|\de_i\de_j\xi(0)|^2=\frac13 K''''(0), \quad \E\de_i^2\xi\de_j^2\xi=\frac13 K''''(0),
\ee
where $i,j,h,k$ are any $4$ distinct indices and $K''''(0)$ denotes the fourth derivative of $K$ evaluated at the origin. In particular, if $d=2$, then $H=H_0\xi$ satisfies \eqref{eq:rothessian} with the additional condition that $c=\sigma^2$:
\be\label{eq:traslhessian}
\begin{pmatrix}
h_1\\
h_2\\
h_{3}
\end{pmatrix}
\sim \mathcal{N} \left ( 0, K''''(0)\begin{pmatrix}
1  &    \frac13            &  0    \\
\frac13            & 1  &  0    \\
0            &   0             & \frac13   
\end{pmatrix}          \right ) .
\ee
and 
\be
\xi(0)=-\Delta\xi(0)\frac{3}{4 K''''(0)}+\gamma_0 \sqrt{1-\frac{3}{2 K''''(0)}}.
\ee
for some $\gamma_0\sim \m N(0,1)$ independent from $H$, $\Delta$ denoting as usual the Laplacian operator.
\end{prop}
\begin{remark}
The above proposition is in accordance with \cite[Eq. (2.11)]{NICOLAESCU20173412}. Note that our setting includes the Hessian of Berry's random field, for which \cite[Prop. B.6]{NICOLAESCU20173412} does not hold.
\end{remark}
\begin{remark}
For $\xi$ to be a \emph{normal} field on $\R^d$ with respect to the standard metric (see \cref{sec:Setting}), we must have that $K(0)=-K''(0)=1$.
\end{remark}
\begin{proof}
$K(t)$ is an even function of class $\mC^2$, so its Taylor expansion is a series in $t^2$. Let us define $K(t)=h(t^2)$, then or all $n\in \N$
\be 
h^{(n)}(0)=\frac{n!}{(2n)!}K^{(2n)}(0).
\ee
Now, it is sufficient to compute $\de_i\de_j\de_{i'}\de_{j'}h(|x-y|^2)$ at $x=y=0.$ We report only the computation of $\E\de_i^2\xi\de_j^2\xi$.
\bega 
\E\de_i^2\xi\de_j^2\xi &=\frac{d^2}{dt^2}\Big|_0\frac{d^2}{ds^2}\Big|_0 \E\qwe \xi(t,0)\xi(0,s)\ewq
\\
&=
\frac{d^2}{dt^2}\Big|_0\frac{d^2}{ds^2}\Big|_0 h\tyu t^2+s^2\uyt
=
\frac{d^2}{dt^2}\Big|_0\frac{d}{ds}\Big|_0 h'\tyu t^2+s^2\uyt2s
\\
&=
\frac{d^2}{dt^2}\Big|_0 h'\tyu t^2\uyt2=4h''(0)=4\frac{2!}{4!}K''''(0).
\eega
\end{proof}
\begin{prop}\label{prop:berryHess}
   If $\xi: \R^2\randto \R$ is the  Berry random field, with covariance $\E\kop \xi(x)\xi(y)\pok=J_0(\sqrt{2}|x-y|)$\footnote{This normalization, with the factor $\sqrt{2}$, is the only one that ensures that we are in the setting of this paper, namely $K''(0)=-1$. Then, $\xi$ satisfies the almost sure equation $\Delta \xi=-2\xi$.}, then $\xi$ is a normal field on $\R^2$, with $K^{(iv)}(0)=\frac 32$ and
   \be 
   H_0\xi=\begin{pmatrix}
\xi_1''(0) & \xi_{12}''(0)
\\
\xi_{12}''(0) & \xi_2''(0)
   \end{pmatrix}
   \quad
   \text{with}
\quad
\begin{pmatrix}
\xi_1''(0) \\ \xi_2''(0) \\ \xi_{12}''(0)
\end{pmatrix}\sim \mathcal{N}\tyu 0, \frac{3}{2}\begin{pmatrix}
1  &    \frac13            &  0    \\
\frac13            & 1  &  0    \\
0            &   0             & \frac13   
\end{pmatrix}\uyt.
   \ee
In particular, $\xi(0)=-\frac{1}{2}\Delta\xi(0)$.
\end{prop}
\begin{prop}\label{prop:BFHess}
   If $\xi: \R^2\randto \R$ is the  Bargmann-Fock random field, with covariance $\E\kop \xi(x)\xi(y)\pok=e^{-\frac{|x-y|^2}{2}}$, then $\xi$ is a normal field on $\R^2$, with $K''''(0)=2$ and
   \be 
   H_0\xi=\begin{pmatrix}
\xi_1''(0) & \xi_{12}''(0)
\\
\xi_{12}''(0) & \xi_2''(0)
   \end{pmatrix}
   \quad
   \text{with}
\quad
\begin{pmatrix}
\xi_1''(0) \\ \xi_2''(0) \\ \xi_{12}''(0)
\end{pmatrix}\sim \mathcal{N}\tyu 0, 2\begin{pmatrix}
1  &    \frac13            &  0    \\
\frac13            & 1  &  0    \\
0            &   0             & \frac13   
\end{pmatrix}\uyt.
   \ee
In particular, $\xi(0)=-\frac{3}{8}\Delta\xi(0)+\frac12 \gamma_0$.
\end{prop}
\subsubsection{Isotropic spherical fields}
\begin{prop}\label{prop:spherhess}
Take $\xi: \S^2\randto \R$ to be a Gaussian isotropic spherical random field, with covariance $\E\kop \xi(p)\xi(q)\pok=K(\theta(p,q))$, where $\theta(p,q)$ denotes the spherical distance of $p$ and $q$; let $K(0)=1$, $a^2:=K''''(0)$ and $r^2:=-K''(0)$. Then, $\hat \xi(p):=\xi(r^{-1}p)$ is a normal field on $r\S^2$. In such case, for any fixed $p\in \S^2$ and orthonormal basis $u,v$ of $p^\perp$ we have that the Riemannian Hessian of $\xi$ has the following law:
   \be 
   H_p\hat\xi=\begin{pmatrix}
\xi''_u(p) & \xi_{uv}''(p)
\\
\xi_{uv}''(p) & \xi_v''(p)
   \end{pmatrix}
   \quad
   \text{with}
\quad
\begin{pmatrix}
\xi_u''(p) \\ \xi_v''(p) \\ \xi_{uv}''(p)
\end{pmatrix}\sim \mathcal{N}\tyu 0, \frac{a^2}{r^4}\begin{pmatrix}
 1 &    \frac13-\frac2{3 a^2}          &  0    \\
\frac13-\frac2{3a^2}         & 1  &  0    \\
0            &   0             & \frac13+\frac1{3a^2}  
\end{pmatrix}\uyt.
   \ee
In particular, in the notation of \cref{lem:rothessional} $c-\sigma^2=-\frac{1}{r^4}$ and $   H_p\hat\xi$ is Hessian-like with respect to
\be
\gamma=-\Delta_{S^2}\xi(0)\frac{3r^4}{(4a^2-2)}+\gamma_0 \sqrt{1-\frac{3r^4}{(2a^2-1)}},
\ee
where $\Delta_{\S^2}$ denotes the spherical Laplacian and $\gamma_0\sim \m N(0,1)$ is independent from $   H_p\hat\xi$.
\end{prop}
\begin{proof}
Notice that $K(\theta)=h(\cos \theta)$ for some $\mC^1$ function $h$.
If $\xi$ is isotropic, then it has rotation invariant Hessian, thus \cref{lem:rothessional} holds. For this reason it is enough to compute $\E |\xi_u''(p)|^2$ and $\E [\xi_u''(p)\xi_v''(p)]$. Let $p(\theta,\phi)\in \S^2$ be the point with polar coordinates $\theta$ and $\f$ and let us assume that $p=p(0,\f)$ is the north pole, so that the curves $t\mapsto p(t,\f)$ are geodesics, for any fixed $\f$.
\bega 
\E [\xi_u''(p)\xi_v''(p)]&=
\frac{d^2}{dt^2}\Big|_0\frac{d^2}{ds^2}\Big|_0 \E[\hat\xi(rp(r^{-1}t,0))\hat\xi(rp(r^{-1}s,\frac{1}{2}\pi)]
\\
&=\frac{d^2}{dt^2}\Big|_0\frac{d^2}{ds^2}\Big|_0 \E[\xi(p(t,0))\xi(p(s,\frac{1}{2}\pi)]r^{-4}
\\
&=
\frac{d^2}{dt^2}\Big|_0\frac{d^2}{ds^2}\Big|_0 h\tyu \langle p(t,0),p(s,\frac{1}{2}\pi)\rangle \uyt r^{-4}
\\
&=\frac{d^2}{dt^2}\Big|_0\frac{d^2}{ds^2}\Big|_0 h\tyu \cos t \cos s \uyt=h''(1)+h'(1)=\tyu \frac13 K''''(0)-\frac{2}{3}K''(0)\uyt  r^{-4}.
\eega
An analogous computation shows that $\E |\xi_u''(p)|^2=K''''(0)  r^{-4}$.
\end{proof}
\begin{remark}
It is well-known that under isotropy the covariance function can be expressed as 
\[
\E\kop \xi(p)\xi(q)\pok=K(\theta(p,q))=\sum_{\ell}\frac{2\ell +1}{4\pi}C_{\ell}P_{\ell}(\cos\theta(p,q)),
\]
the non-negative sequence $(C_{\ell})_{\ell=0,1,2,...}$ denoting the angular power spectrum of the field and $P_{\ell}(.)$ representing Legendre polynomials (see e.g. \cite{MarinucciPeccati11}). Then standard computations yield (see e.g. \cite{CMW16})
\[
K(0)=\sum_{\ell}\frac{2\ell +1}{4\pi}C_{\ell},
\]
\[
K''(0)=-\sum_{\ell}\frac{2\ell +1}{4\pi}\frac{\lambda_{\ell}}{2}C_{\ell}\ ,
\]
where we wrote $\lambda_{\ell}=\ell(\ell+1)$, and
\[
K''''(0)=\sum_{\ell}\frac{2\ell +1}{4\pi}\left ( 3\frac{\lambda_{\ell}(\lambda_{\ell}-2)}{8}+\frac{\lambda_{\ell}}{2} \right ) C_{\ell} \ .
\]
\end{remark}

\subsection{Behavior of the Hessian under scaling limit}
\begin{prop}
Assume that $X_\lambda\colon M\randto \R$ is a sequence of $\mC^2$ GRFs of unit variance with Adler-Taylor metric $g^\lambda
$ (so that $X_\lambda\sim \mathcal{N}(M,g^\lambda)$) and such that the following limit holds:
\be\label{eq:scaling} 
X_\lambda \tyu \exp_p^{g_\lambda}\tyu \frac{u}{\sqrt{\lambda}}\uyt\uyt \xrightarrow[\lambda\to +\infty]{\mC^2(T_pM)} \xi(u),
\ee
in the space of $\mC^2$ functions of $u\in T_pM$, where $\xi\colon T_pM\randto \R$ is some GRFs on $T_pM\cong \R^m$. Let $H^{g^\lambda}$ denote the Hessian operator with respect to the metric $g^\lambda$. Then, we have 
that for every $p\in M$,
\be 
H_p^{g^\lambda}X_\lambda(u,v) \sim_{\lambda\to +\infty} \lambda H_0\xi(u,v).
\ee
\end{prop}
\begin{proof}
It is enough to check the limit for $u=v$, since the symmetric form $H_pX_\lambda(u,v)$ can be recovered from the quadratic form $H_pX_\lambda(u,v)$ by means of the polarization formula. We have that 
\be 
\frac{1}{\lambda}H^{g^\lambda}_pX_{\lambda}(u,u)
=
\frac{d^2}{dt^2}\Big|_{t=0}X_\lambda \tyu \exp_p^{g_\lambda}\tyu \frac{tu}{\sqrt{\lambda}}\uyt\uyt
\xrightarrow[\lambda\to +\infty]{\R} \frac{d^2}{dt^2}\Big|_{t=0}\xi(tu).
\ee
\end{proof}
The hypotheses of the above proposition are, in particular, satisfied for Gaussian Laplace eigenfunctions: $\Delta_{\S^m}X_\lambda=-\lambda X_\lambda$ on the sphere $M=\S^m$, with $\lambda\in \{ \ell(\ell+m-1)\colon \ell\in 
\N\}$ tending to $+\infty$. In this case, we have that $g^\lambda=\frac
\lambda m g$, where $g$ is the standard round metric on $\S^m$ and therefore $\exp_x^g=\exp_x^{g^\lambda}$, so that \cref{eq:scaling} is the usual scaling limit, with $\xi$ being Berry's random field on $T_p\S^m\cong \R^m$. 

In this situation, as $\lambda\to +\infty$, we can approximate:
\be 
E_k^t\tyu g^\lambda, H^{g^\lambda}_pX_\lambda\uyt=E_k^t\tyu g, \frac{m}{\lambda}H^{g^\lambda}_pX_\lambda\uyt\sim E^t_k(g,m H_0\xi),
\ee
which by \cref{thm:one} gives as $\lambda\to +\infty$
\be 
\E\#C_t^\lambda\sim \E\kop\frac{1}{\chi_k^m}\pok\frac{1}{(2\pi)^{\frac m2}} \vol(M) \cdot E^t_k(g,m H_0\xi)\cdot \tyu\frac{\lambda}{m}\uyt^{\frac{m}{2}}.
\ee

\section{An exact formula for local maxima} \label{sec:asymp}


For applications in Statistics, Mathematical Physics and Machine Learning it is of course very common to focus on local maxima, especially at high threshold. These are the random quantities that must be considered, for instance, when probing for galactic point sources among CMB polarisation data, or when investigating the convergence properties of statistics and machine learning optimization algorithms. In this Section, we show how much more explicit results can be obtained, in the limit of high thresholds $u$.

Let us first introduce the following auxiliary Gaussian random function $\f\randin \mC^\infty(M\times \S^{k-1})$
\bega
\f:M\times \S^{k-1}\to \R, \quad \f(p,u):=Y(p)^Tu.
\eega
Indeed, we have that if $F=\frac{|Y|^2}{2}$ has non-degenerate maxima (true a.s.), then there is a bijection:
\bega 
\kop p\in M: \text{ local maxima of $F$}\pok
&\xrightarrow{\cong}\kop (p,v)\in M\times \S^{k-1}: \text{ local maxima of $\f$}\pok
\\
p&\to \tyu p, \frac{Y(p)}{|Y(p)|}\uyt.
\eega
Notice also that $F(p)=\frac{1}{2}\f\tyu p, \frac{Y(p)}{|Y(p)|}\uyt^2.$ It follows that almost surely we have that, for all $t\ge 0$
\bega 
C_t\cap \{p\in M: H_pF<0\}&=\kop p\in M: \text{ local maxima of $F$, with value $\ge \frac{t^2}{2}$}\pok
\\
&\cong\kop (p,v)\in M\times \S^{k-1}: \text{ local maxima of $\f$ of value $\ge t$}\pok
\\
&=\kop (p,v)\in M\times \S^{k-1}: d_{(p,v)}\f=0, d^2_{(p,v)}\f<0, \f(p,v)\ge t\pok
\\&=: C_t^{\mathrm{Max}}
\eega
Recall that $Y=(X^1,\dots,X^k),$ where $X^i\sim X$ are i.i.d. copies of $X\sim \mathcal{N}(M,g)$.
\begin{lemma}
Let $(p,u)\in M\times \S^{k-1}$ and let us choose orthonormal bases to identify $T_pM\cong\R^m$ and $T_u\S^{k-1}\cong \R^{k-1}$. Then $d_pX$ is identified with a  standard Gaussian \emph{row} in $\R^m$ and $H_pX$ is identified with a symmetric Gaussian $m\times m$ matrix.  The $2$-jet of $\f$ has the following joint distribution:
\bega 
\f(p,v)&=X^1(p) 
\randin \R
\\
d_{(p,v)}\f&=\tyu d_pX^1,X^2(p),\dots , X^k(p) \uyt \randin \R^{m}\times \R^{k-1}
\\
H_{(p,v)}\f &=\begin{pmatrix}
    H_pX^1 & (d_pX^2)^T & \dots & (d_pX^k)^T \\
    d_pX^2  \\
    \dots & & -X^1(p)\mathbbm{1}_{k-1} &
    \\
    d_pX^k  \\
\end{pmatrix}\randin \tyu \R^{m}\times \R^{k-1}\uyt \otimes \tyu \R^{m}\times \R^{k-1}\uyt. 
\eega
In particular, we have that $d_{(p,v)}\f$ and $H_{(p,v)}\f$ are independent and the dependence between $\f(p,v)$ and $H_{(p,v)}\f$ is that $\E\{H_{(p,v)}\f \cdot X^1(p)\}=-\mathbbm{1}_{m+k-1}$.
\end{lemma}
\begin{proof}
    The result is the same (in law) for all $v\in \S^{k-1},$ so that we can choose $v=e_1.$ Then $T_{v}\S^{k-1}$ is identified with $e_1^{\perp}=\R^{k-1}.$ 
    In a neighborhood of $e_1$ in $\S^{k-1},$ we take affine coordinates $u=u^2,\dots u^k\in \R^{k-1},$ to parametrize the point $v(u)=(\sqrt{1-|u|^2},u)\in \S^{k-1},$ so that $\frac{d}{dt}(v(u))\in T_{e_1}\S^{k-1}$ is isometrically identified with $\dot u\in \R^{k-1}$. Then, we have $\f(p,v)=Y(p)^Tv(u).$ For every $\dot p\in T_pM,$ and $\dot u\in \R^{k-1},$ we compute the Hessian as follows. Let $p(t)$ be a geodesic in $M$ such that $p(0)=p$ and $\dot p(0)=\dot p$ and let $u(t)$ parametrize a geodesic $v(u(t))$ on $\S^{k-1},$ with $u(0)=0$ and $\dot u(0)=\dot u$, that is, $u(t)=\tyu\sin(t|\dot u|)\frac{\dot u}{|\dot u|}\uyt$. Then,
    \bega
H_{(p,e_1)}\f\tyu(\dot p, \dot u),(\dot p, \dot u)\uyt &=\frac{d^2}{dt^2}\qwe 
X^1(p)\sqrt{1-|u|^2}+
\sum_{i=2}^k X^i(p)u^i \ewq
\\
&= H_pX^1(\dot p, \dot p)+X^1(p) \frac{d^2}{dt^2} 
\sqrt{1-|u|^2}+2\sum_{i=2}^k d_pX^i(\dot p)\dot{u}^i 
\\
&=H_pX^1(\dot p, \dot p)-X^1(p) |\dot u|^2+2\sum_{i=2}^k d_pX^i(\dot p)\dot{u}^i.
    \eega
\end{proof}
Observe that the law above depends uniquely on the metric $g$ at $p$, that is essentially the covariance of $d_pX$, and on the law of $H_pX$. 
\begin{lemma}\label{lem:jetfi}
Let $(p,u)\in M\times \S^{k-1}$ and let us choose orthonormal bases to identify $T_pM\cong \R^m$ and $T_u\S^{k-1}\cong \R^{k-1}$. The $2$-jet of $\f$ has the following joint distribution:
\bega 
\f(p,v)&=\gamma_1
\randin \R
\\
d_{(p,v)}\f&=(\gamma_{1,1},\dots, \gamma_{m,1},\gamma_2,\dots, \gamma_{k}) \randin \R^{m}\times \R^{k-1}
\\
H_{(p,v)}\f &=\begin{pmatrix}
               &       &            &  \gamma_{1,2} & \dots & \gamma_{1,k}  \\
               & H_pX  &            &             & \dots &             \\
               &       &            &  \gamma_{m,2} & \dots & \gamma_{m,k}  \\
    \gamma_{1,2} & \dots & \gamma_{m,2} &                                   \\
               & \dots &            &             &-\gamma_1\mathbbm{1}_{k-1}      &             \\
    \gamma_{1,k} & \dots & \gamma_{m,k}                                     \\
\end{pmatrix}\randin \tyu \R^{m}\times \R^{k-1}\uyt \otimes \tyu \R^{m}\times \R^{k-1}\uyt. 
\eega
where $\gamma_{i,j}\sim \mathcal{N}(0,1)$ are i.i.d. and independent from $(\gamma_1, H_pX)$ and $\E\{\gamma_1 H_pX\}=-\mathbbm{1}_m.$ In particular, the above law is invariant under orthonormal changes of basis in $T_pM$. 
\end{lemma}
\begin{remark}\label{rem:finorm}
The above lemma shows that $\f$ is a normal field on $M\times \S^{k-1}$.
\end{remark}
\begin{definition}\label{def:D}
 Let $(H,\gamma_1)$ be as in \cref{def:H}: $H$ is an $m\times m$ Hessian-like Gaussian matrix and $\E\{\gamma_1 H\}=-\mathbb{1}_m$. Let $\tilde{H}$ be distributed as $ H_{(p,v)}\f$ in \cref{lem:jetfi}, that is:
 \be 
 \tilde H =\begin{pmatrix}
               &       &            &  \gamma_{1,2} & \dots & \gamma_{1,k}  \\
               & H  &            &             & \dots &             \\
               &       &            &  \gamma_{m,2} & \dots & \gamma_{m,k}  \\
    \gamma_{1,2} & \dots & \gamma_{m,2} &                                   \\
               & \dots &            &             &-\gamma_1\mathbbm{1}_{k-1}      &             \\
    \gamma_{1,k} & \dots & \gamma_{m,k}                                     \\
\end{pmatrix},
 \ee
 where $\gamma_{i,j}\sim \mathcal{N}(0,1)$ are i.i.d. and independent from $(\gamma_1, H)$. Let us use the notation $\mathfrak{G}(m)\subset \R^{m\times m}$ to denote the subset of positive definite symmetric matrices. \footnote{In general, the space of positive definite symmetric matrices is the space of scalar products, so we prefer to work with that instead than with the set of negative definite matrices. Of course, the two are identical.} Define
\be 
D^t_k(H):=\E \qwe | \det\tyu \tilde{H} \uyt|\cdot  1_{\mathfrak{G}({m+k-1})}(-\tilde{H})1_{[t,+\infty)}(\gamma_1)\ewq.
\ee
\end{definition}

We can now exploit the previous expressions to derive an explicit formula for the critical values of chi fields.

\begin{thm}\label{thm:max}
    For any $A\subset M,$ we have that
    \be
\E\kop \#\tyu C_t^{\mathrm{Max}}\cap A \uyt\pok
=
\frac{\vol(\S^{k-1})}{(2\pi)^{\frac{m+k-1}{2}}}\int_A D^t_k\tyu[H_pX]\uyt dM(p).
    \ee
\end{thm}
\begin{proof}
    We apply the Alpha-Kac-Rice formula (see \cite{MathiStec}) to $\f:M\times \S^{k-1}\randto \R$, with 
    \be 
    \a(\f,p,v)=1_{\mathfrak{G}({m+k-1})}(-H_{(p,v)}\f)\cdot  1_{[t,+\infty)}(\f(p,v)).
    \ee
    \cite[Prop. 4.10]{MathiStec} shows that we can, because $d\f$ is Gaussian and $d_{(p,v)}\f$ is non-degenerate for all $(p,v)\in M\times \S^{k-1}.$ 
    Since $d_{(p,v)}\f$ and $\a(\f,p,v)$ are independent, the formula says that
    \bega 
\E\kop \#\tyu C_t^{\mathrm{Max}}\cap A \uyt\pok
&= \E\kop \sum_{(p,v)\in A\times \S^{k-1}\text{ s.t. }d_{(p,v)}\f=0} \a(\f,p,v)\pok 
\\&=
\int_{\S^{k-1}} \int_{A}\E \kop | \det\tyu H_{(p,v)}\f \uyt|\cdot  \a(\f,p,v)\pok\rho_{[d_{(p,v)}\f]}(0) dM(p) d\S^{k-1}(v)
    \eega
    Observe that, by \cref{lem:jetfi}, the density of $d_{(p,v)}\f$ at zero is equal to 
    \be 
    \rho_{[d_{(p,v)}\f]}(0)=\frac{1}{(2\pi)^{\frac{m+k-1}{2}}}
    \ee 
    and that the expectation term is exactly $D^t_k(H_pX),$ which is constant in $v$, hence we conclude.
\end{proof}

\section{High-threshold asymptotics}\label{sec:HighTresh}

Let us try to get more explicit formulae in the case $m=k=2$, and $X$ is an isotropic Gaussian field on $M$, and $M$ is the round sphere or the plane. In this case, we can write

\be 
 \tilde H =\begin{pmatrix}
               h_1      &  h_2          & \gamma_1  \\
               h_2     &  h_3          & \gamma_2  \\
               \gamma_1 & \gamma_2      & -\gamma    \\
\end{pmatrix},
 \ee
where $(\gamma_1,\gamma_2)$ is independent from $(h_1,h_2,h_3,\gamma)$ with 
\be
\begin{pmatrix}
\gamma_1\\
\gamma_2\\
\end{pmatrix}
\sim \mathcal{N} \left ( \begin{pmatrix}
0\\
0\\
\end{pmatrix}, \begin{pmatrix}
1 & 0 \\
0 & 1 \\
\end{pmatrix}          \right ),
\ee
and, since this fields have rotation invariant Hessian in the sense of \cref{def:rothessian}, we have from \cref{lem:rothessional} that there are constants $\sigma,c \ge 0$ such that
\be
\begin{pmatrix}
h_1\\
h_3\\
h_2\\
\gamma
\end{pmatrix}
\sim \mathcal{N} \left ( \begin{pmatrix}
0\\
0\\
0\\
0
\end{pmatrix}, \begin{pmatrix}
2\sigma^2+c   &    c          &  0       &    -1\\
c          & 2\sigma^2+c & 0        &     -1 \\
0          &   0           & \sigma^2 & 0   \\
-1         &   -1           &    0    & 1
\end{pmatrix}          \right ) .
\ee


In fact, one can compute $c,\sigma^2$ by a universal formula depending only on the fourth derivative of the covariance of the field and on the model chosen: \cref{prop:stathess} for the plane and \cref{prop:spherhess} for the sphere.
It follows from \cref{prop:covjet} that the vector $(\tilde{h}_1,\tilde{h}_2,h_3)$ is zero mean and independent from $\gamma$, where $\tilde{h_i}=h_i+\gamma$. With this notation, we have

\be 
 \det(\tilde H) =\det\begin{pmatrix}
               \tilde{h}_1-\gamma      &  h_2          & \gamma_1  \\
               h_2      &  \tilde{h}_3-\gamma          & \gamma_2  \\
               \gamma_1 & \gamma_2      & -\gamma    \\
\end{pmatrix},
 \ee
\be
=-\tilde{h}_1\tilde{h}_3\gamma+h_2^2\gamma-\gamma^2(\tilde{h}_1+\tilde{h}_3)-\gamma^3-\gamma_1^2\tilde{h}_3+\gamma\gamma_1^2-\gamma_2^2\tilde{h}_1+\gamma^2_2\gamma+2h_2\gamma_2\gamma_1 \ .
\ee
\begin{remark}
It is easy to see that the expected value of this determinant in the region where $(\tilde{h}_1,h_2,\tilde{h}_3)$ is in $\mathbb{R}^3$, $(\gamma_1,\gamma_2)$ is in $\mathbb{R}^2$ and $\gamma \geq t$, is equal to 
\[
A_1:=\int_t^{\infty} \frac{1}{\sqrt{2\pi}} \exp\tyu -\frac{\gamma^2}{2}\uyt (\gamma^3-{  (3-c+\sigma^2)}\gamma)d\gamma=(H_2(t)+(c-\sigma^2))\varphi(t) \ .
\]
The idea that we will follow in this section is to show that the difference between this term and the one with the absolute value is of smaller order in $t$. 
\end{remark}
\begin{remark}
It should be noted that, for $m=k=2$,
\[
D^t_k(H):=\E \qwe | \det\tyu \tilde{H} \uyt|\cdot  1_{\mathfrak{G}(3)}(-\tilde{H})1_{[t,+\infty)}(\gamma)\ewq
=\E \qwe  \det\tyu -\tilde{H} \uyt\cdot  1_{\mathfrak{G}(3)}(-\tilde{H})1_{[t,+\infty)}(\gamma)\ewq
\]
(because the determinant of $\tilde H$ is necessarily negative)
\[
=\E \qwe  \det\tyu -\tilde{H} \uyt\cdot  1_{[t,+\infty)}(\gamma)\ewq
-\E \qwe  \det\tyu -\tilde{H} \uyt\cdot  \tyu 1-1_{\mathfrak{G}(3)}(-\tilde{H})\uyt 1_{[t,+\infty)}(\gamma)\ewq
\]
\[
=A_1+A_2 \ .
\]
\end{remark}
We are therefore able to establish the following result.
\begin{thm}\label{thm:mainbody}
Let $m=k=2$ and let $X$ have rotation invariant Hessian, in the sense of \cref{def:rothessian}. As $t \rightarrow +\infty$, we have that
\[
\frac{D^t_2(H)}{{  \tyu H_2(t)+(c-\sigma^2) \uyt}\phi(t)}=\frac{\E \qwe | \det\tyu \tilde{H} \uyt|\cdot  1_{\mathfrak{G}({3})}(-\tilde{H})1_{[t,+\infty)}(\gamma)\ewq
}{{  \tyu H_2(t)+(c-\sigma^2) \uyt}\phi(t)}=1+O\tyu\exp(-\delta t^2)\uyt \ .     
\]
\end{thm}

\begin{proof}
Note first that
\[
A_1=\E \qwe  \det\tyu -\tilde{H} \uyt\cdot 1_{[t,+\infty)}(\gamma)\ewq
\]
\be
=\E \qwe (\tilde{h}_1\tilde{h}_3\gamma-h_2^2\gamma+\gamma^2(\tilde{h}_1+\tilde{h}_3)+\gamma^3+\gamma_1^2\tilde{h}_3-\gamma\gamma_1^2+\gamma_2^2\tilde{h}_1-\gamma^2_2\gamma-2h_2\gamma_2\gamma_1) 
 \cdot 1_{[t,+\infty)}(\gamma) \ewq 
\ee
\[
=\E 
\qwe (\gamma^3-
{  (3-c+\sigma^2)}\gamma) 
 \cdot 1_{[t,+\infty)}(\gamma) 
 \ewq 
 =
 \int_t^{\infty} \tyu H_3(x){  +(c-\sigma^2)H_1(x)} \uyt \phi(x)dx
\]
\[
= {  \tyu H_2(t)+(c-\sigma^2) \uyt}\phi(t) \ ,
\]
because
\[
{  \E[h_2^2]=\sigma^2,\ \E[\tilde h_1\tilde h_3]=c-1},\quad \E[\gamma_1^2]=\E[\gamma_2^2]=1 \ , \ \E[\tilde{h}_1]=\E[\tilde{h}_3]=\E[h_2]=0 \ ,
\]
and using one of the defining property of Hermite polynomials, saying that $\int_t^\infty H_{n+1}(\gamma)\phi(\gamma)d\gamma=H_n(t)\phi(t)$.
Now let us focus on \[
A_2:=-\E \qwe  \det\tyu -\tilde{H} \uyt\cdot  \tyu 1-1_{\mathfrak{G}({3})}(-\tilde{H})\uyt 1_{[t,+\infty)}(\gamma)\ewq \ ;
\]
in the above integral, the Hessian must be negative definite, which implies that some of the mixed products involving $h$'s and $\gamma$ must be larger than $\gamma^3$; we shall show that this probability is exponentially small in the regime where $\gamma>t$ and $t$ grows to infinity. 
Precisely, observe that $\tilde H$ is not negative definite if and only if there exists a vector $\lambda\in \R^3$, such that $\lambda^T \tilde H \lambda \ge 0$. Let $\mu:=\max\{|\tilde h_1|,|h_2|,|\tilde h_3|,|\gamma_1|,|\gamma_2|\}$ and observe that if $\mu < \frac{3}{8}t$ and $\gamma\ge t$, then we for all $\lambda=(\lambda_1,\lambda_2,\lambda_3)$ with $\|\lambda\|\le 1$ we have that
\bega 
\lambda^T \tilde H \lambda 
&=(\tilde h_1 -\gamma)\lambda_1^2+(\tilde h_3 -\gamma)\lambda_2^2+(-\gamma)\lambda_3^2
+2h_2\lambda_1\lambda_2+2\gamma_1\lambda_1\lambda_2+2\gamma_2\lambda_1\lambda_2
\\
&\le 8\mu-3t<0.
\eega
Therefore, we can deduce the following.
\be\label{eq:indicatrici} 
\tyu 1-1_{\mathfrak{G}({3})}(-\tilde{H})\uyt 1_{[t,+\infty)}(\gamma) 
\le  
1_{[\frac{3t}{8},+\infty )}(\mu)1_{[t,+\infty)}(\gamma).
\ee

On the other hand, it is easy to see that for any $\gamma>0$, 
\be \label{eq:updet}
|\det\tyu -\tilde{H} \uyt| 
\le 
11 (\gamma^3+\mu^3).
\ee
Now, $\mu^3$ is bounded by the multiple of a chi-distributed random variable, meaning that $\mu\le C \chi_5$, for a constant $C>0$ (which should be comparable with $\sqrt{\sigma^2+c}$) and a chi-distributed random variable $\chi_5$ of parameter $5$, independent from $\gamma$. It follows from \cref{lem:Ointeexp} below that, by combining \cref{eq:indicatrici} and \cref{eq:updet}, that the corresponding expected value $A_2$ is bounded above as follows
\bega
A_2&\le 
11\int_{\frac{3t}{8C}}^{\infty}\int_t^{\infty}(\chi^3+\gamma^3)\chi^4 \exp(-\chi^2/2) \exp(-\gamma^2/2) d\gamma d\chi
\\
&= O\tyu \tyu \int_t^\infty (t^6+t^3 \gamma^3)\exp(-\gamma^2/2)d\gamma\uyt\exp\tyu -2\delta t^2\uyt\uyt
\\
&= O\tyu  (t^6+t^3\cdot t^2)\exp\tyu -\frac{t^2}{2} \uyt\exp\tyu -2\delta t^2 \uyt\uyt
\\
&= O\tyu \exp\tyu -t^2\tyu\frac{1}{2}+\delta\uyt\uyt\uyt.
\eega
for some small constant $\delta >0$, depending on $\sigma$ and $c$.
This integral is exponentially smaller than the leading term. 
\end{proof}
\begin{lemma}\label{lem:Ointeexp}
As $t\to +\infty$, we have 
\be 
\int_{\frac{t}{C}}^{+\infty} x^n \exp\tyu-\frac{x^2}{2}\uyt dx = O\tyu t^{n-1}\exp\tyu -\frac{t^2}{2 C^2} \uyt\uyt.
\ee
\end{lemma}
\begin{proof}
The proof is straightforward and hence omitted.
\end{proof}
By combining the latter with \cref{thm:max}, we obtain the following, which proves also our main result, \cref{thm:two}.
\begin{cor}\label{cor:asyM}
In the same setting as above, with $m=k=2$ when $X$ is isotropic, or whenever the value of $D_k^t([H_pX])$ is constant in $p$, we have that as $t\to +\infty$
\be
\E\kop \#\tyu C_t^{\mathrm{Max}} \uyt\pok
\cdot
\tyu 
\frac{1}{(2\pi)^{\frac{1}{2}}
}
\vol(M)
\tyu H_2(t)+(c-\sigma^2) \uyt\phi(t)
\uyt ^{-1}
=
1+O\tyu \exp(-\delta t^2)\uyt.
\ee
\end{cor}
Recall that, by \cref{prop:spherhess}, for $X$ an isotropic field on $\S^2$, we have $c-\sigma^2=-\frac 1 {r^2}$, therefore the above results implies \cref{MainTh}.


\subsection{Proof of \cref{LastMainTh}}
The first statement of \cref{LastMainTh} is \cref{thm:max}, so it remains to show the validity of the asymptotic equivalences. We will address and justify each one of them in the following. The idea is to exploit the fact that $\f$ is a normal field (in the sense of \cref{def:norm} on $M\times \S^{k-1}$, by \cref{rem:finorm}.
\subsubsection{The connection with the Euler-Poincaré Characteristic and excursion probabilities}
By showing that in the high-threshold limit the dominant term corresponds to the expected value of the determinant without the modulus, we are actually proving that the number of maxima taking values larger than $t$ is asymptotically equivalent to the Euler-Poincar\'e characteristic for the excursion set of $\f\colon \S^2\times \S^1\to \R$ which we introduced in Section 5, that is
\bega\label{eq:MorseThm} 
\chi(\f\ge t)\overset{\text{\cite[Cor. 9.3.5]{AdlerTaylor}}}{=}\sum_{\{(p,v)\in r\S^2\times \S^1 :\ d_{(p,v)}\f=0\}} \mathrm{sgn}\tyu \det \tyu -H_{(p,v)}\f\uyt\uyt \cdot 1_{[t,+\infty]}(\f(p,v)),
\eega
since $(-1)^{\text{index}(-H)}=\mathrm{sgn}(\det(-H))$.
Indeed the expectation of the determinant without the absolute value, i.e., the term $A_1$ in the proof of \cref{thm:mainbody}, is indeed the Kac-Rice density of the right-hand side. Hence our result \cref{MainTh} is equivalent to the following limit:
\be 
\tyu  
 \  \E\chi(\f\ge t)\
\uyt^{-1} 
\mathbb{E}[\mu_t(\mathbb{S}^2,f_2)]
=
1 + O \tyu\exp\tyu -\delta t^2\uyt\uyt \ ,
\ee
Indeed, the Adler-Taylor formula for $\E\chi(\f\ge t)$ yields
\bega 
\E\chi(\f\ge t)&\overset{\text{\cite[Th. 12.4.1]{AdlerTaylor}}}{=}
\m L_0(r\S^2\times \S^1)\rho_0(t)+\m L_1(r\S^2\times \S^1)\rho_1(t)
\\
&\quad +\m L_2(r\S^2\times \S^1)\rho_2(t)+\m L_3(r\S^2\times \S^1)\rho_3(t)
\\
&=
0+\m L_1(r\S^2\times \S^1)\cdot \frac{1}{\sqrt{2\pi}}\phi(t)+0+r^2 \vol(\S^2)\vol(\S^1)\cdot \frac{H_2(t)}{2\pi}\frac{1}{\sqrt{2\pi}}\phi(t)
\\
&=
\tyu \m L_1(r\S^2\times \S^1)+ r^2\vol(\S^2)\vol(\S^1)\frac{H_2(t)}{2\pi}
\uyt \cdot \frac{1}{\sqrt{2\pi}}\phi(t)
\\
&=
\tyu 2H_2(t)r^2-\frac{\m L_1(r\S^2\times \S^1)}{2\pi}
\uyt \cdot (2\pi)^{\frac12}\phi(t)
\\
&=
\tyu 2H_2(t)r^2-2
\uyt \cdot (2\pi)^{\frac12}\phi(t).
\eega
To compute $\m L_1(r\S^2\times \S^1)=4\pi$, for any $r$, we can use \cite[Prop. 3.0.1]{pistolato2024}, together with the fact that $\scal(M\times N)=\scal(M)+\scal(N)$ and $\scal(r\S^2)=\frac{2}{r^2}$, while $\scal(\S^1)=0$.  
This corresponds with our final formula of \cref{MainTh}, derived from \cref{cor:asyM}:
\be 
\tyu 
    2r^2\tyu H_2(t)-\frac{1}{r^2}\uyt \cdot (2\pi)^{\frac12}\phi(t)
\uyt^{-1} 
\mathbb{E}[\mu_t(\mathbb{S}^2,f_2)]
=
1 + O \tyu\exp\tyu -\delta t^2\uyt\uyt \ .
\ee 
We note also that, by \cite[Eq. (14.0.2)]{AdlerTaylor} and for a possibly smaller $\delta>0$, 
\bega 
\E\chi(\f\ge t) &= \P (\max_{\S^2\times \S^1} \f \ge t )+ O \tyu\exp\tyu -\tyu\frac12+\delta\uyt t^2\uyt\uyt
\\
&= \P (\max_{\S^2} f_2 \ge t )+ O \tyu\exp\tyu -\tyu\frac12+\delta\uyt t^2\uyt\uyt,
\eega
since, by construction, $\max_{\S^2\times \S^1} \f=\max_{\S^2} f_2$.

\subsubsection{The connection with Betti numbers}
The Euler characteristic $\chi(E)$ of a manifold with boundary $E$ of dimension $m$ is defined as the alternating sum of its Betti numbers (see \cite{MilnorMorse}), and by \cref{eq:MorseThm} (a classical identity in Morse theory, see \cite{MilnorMorse}) it coincides with the alternating sum of the number of critical points of a Morse function. Specifically, in the setting of \cref{thm:max}, \cite[Th. 5.2]{MilnorMorse}\footnote{The theorem is stated for a compact manifold, but its proof and \cite[Th. 3.5]{MilnorMorse} implies that it can be applied also for the excursion set of the Morse function} yields
\be\label{eq:chiBC} 
\chi(\f\ge t)\overset{\text{def}}{=}\sum_{i=0}^m(-1)^ib_i(\f \ge t)\overset{\text{\cite[Th. 5.2]{MilnorMorse}}}{=}\sum_{i=0}^m(-1)^iC_i(\f \ge t)
\ee
where we denote $C_i(\f \ge t):=\#\{d\f=0, \mathrm{index}(H\f)=i, \f\ge t\}$
and moreover, the weak Morse inequality holds:
\be 
b_i(\f \ge t)\overset{\text{\cite[Th. 5.2]{MilnorMorse}}}{\le} C_i(\f \ge t).
\ee
In particular, $b_0(\f\ge t)$ denotes the number of connected components of the excursion set and $C_0(\f\ge t)$ is the number of local maxima of $\f$ with value exceeding $t$.

By studying the asymptotic behavior the Kac-Rice formulas for $\E C_i(\f\ge t)$, Gayet \cite{GayetExEulMax} showed (in the more general context of stratified manifolds) that 
for $i\ge 1$, we have 
\bega
\E C_i(\f \ge t)
&\overset{\text{\cite[Th. 3.6]{GayetExEulMax}}}{=} O \tyu\exp\tyu -(\frac12+\delta) t^2\uyt\uyt.
\eega
Therefore, up to an exponentially small error in expectation, as $t\to +\infty$, all the critical points of $\f$ in the excursion set $\{\f\ge t\}$ are the local maxima. 

An additional observation of \cite{GayetExEulMax} is that Morse theory implies that
\be\label{eq:ballineq}
b_0(\f \ge t)\overset{\text{\cite[Cor. 2.5]{GayetExEulMax}}}{\le} b_0(\f \ge t,\mathbb B)+\sum_{i=1}^m C_i(\f \ge t),
\ee
where $b_0(\f \ge t,\mathbb B)$ denotes the number of connected components that are homeomorphic to a unit ball $\mathbb B$ of dimension $m$. As a consequence, one deduces that the only Betti number of the excursion set that is asymptotically relevant is $b_0(\f \ge t,\mathbb B)$, see \cite[Thm 5.19]{GayetExEulMax}.
The proof of \cref{eq:ballineq} is the following: if $E\subset \{\f \ge t\}$ is a connected component that contains $k$ critical points that are all of index $0$, then by \cite[Th. 3.5]{MilnorMorse} it follows that the Betti numbers of $E$ are $b_0=k,b_1=0,\dots,b_m=0$, which means that $k=1$ and thus that the flow of $-\nabla\f$ deforms $E$ into a small ball around the only maximum, so that $E$ must be homeomorphic to $\mathbb B$. 
\subsubsection{Asymptotic Equivalences}
Let us also define the \emph{total Betti number} of the excursion set as $b(\f\ge t)=\sum_{i=0}^mb_i(\f\ge t)$ and let us denote the total number of critical points as $C(\f\ge t):=\sum_{i=0}^mC_i(\f\ge t)$.
Putting together the asymptotics in the previous two remarks, we deduce that as $t\to +\infty$ we have 
\bega 
\E\#\tyu C_0(\f\ge t) \uyt
&\sim \E C(\f\ge t)
\sim \E b(\f \ge t)\sim \E\chi(\f\ge t)
\\
&\sim\E b_0(\f\ge t) 
\sim \E b_0(\f\ge t; \mathbbm{B})
\\
&\sim \P \tyu \max_{M\times \S^{k-1}} \f \ge t \uyt
\\
&
\sim \sum_{j=0}^{m+k-1}\frac{\m L_{j}(\S^{k-1}\times M)}{(2\pi)^{\frac{j}2}}H_{j-1}(t)\phi(t)
\eega
with an error of $O \tyu\exp\tyu -(\frac12+\delta) t^2\uyt\uyt$.
The last line being \cite[Th. 12.4.1]{AdlerTaylor}.
Observe also that the set $\{\f\ge t\}$ can be homotopically retracted to $\{(p,Y(p))\colon |Y(p)|\ge t\}$ in $\S^{k-1}\times M$, which is diffeomorphic (being a graph) to the set $\{f_k\ge t\}\subset M$. This implies that the two sets have the same Betti numbers. Moreover, 
a connected component of the former is heomeomorphic to a ball 
if and only if the corresponding connected component of $\{f_k\ge t\}$ is. 
Finally, in the setting of \cref{thm:max}, it is easy to show that critical points of $\f$ correspond with critical points of $f_k$ so that $C(\f\ge t)=\#C_t$ and we already observed that $C(\f\ge t)=\#\tyu C_t^{\mathrm{Max}} \uyt$. Therefore, up to an error $O \tyu\exp\tyu -(\frac12+\delta) t^2\uyt\uyt$, we have the same asymptotic equivalences for $f_k$:
\bega 
\E(\#C_t^{\mathrm{Max}})
&\sim \E\# C_t
\sim \E b(f_k \ge t)\sim \E\chi(f_k\ge t)
\\
&\sim\E b_0(f_k\ge t) 
\sim \E b_0(f_k\ge t; \mathbbm{B})
\\
&\sim \P \tyu \max_{M} f_k \ge t \uyt
\\
&
\sim \sum_{j=0}^{m+k-1}\frac{\m L_{j}(\S^{k-1}\times M)}{(2\pi)^{\frac{j}2}}H_{j-1}(t)\phi(t)
\\
&\overset{\text{($M=\S^2$, $k=2$})}{\sim} (2+2r^2H_2(t))\sqrt{2\pi} \phi(t) .
\eega
The latter asymptotics conclude the proof of \cref{LastMainTh}. 

\bibliographystyle{alpha}
\bibliography{critical.bib}

\newcommand{\etalchar}[1]{$^{#1}$}
\begin{thebibliography}{CCM{\etalchar{+}}24}

\bibitem[ABA13]{Auffinger13}
Antonio Auffinger and Gerard Ben~Arous.
\newblock Complexity of random smooth functions on the high-dimensional sphere.
\newblock {\em Ann. Probab.}, 41(6):4214--4247, 2013.

\bibitem[AD22]{Azais22}
Jean-Marc Aza\"{\i}s and C\'{e}line Delmas.
\newblock Mean number and correlation function of critical points of isotropic
  {G}aussian fields and some results on {GOE} random matrices.
\newblock {\em Stochastic Process. Appl.}, 150:411--445, 2022.

\bibitem[ASZ20]{BenArous20}
G\'{e}rard~Ben Arous, Eliran Subag, and Ofer Zeitouni.
\newblock Geometry and temperature chaos in mixed spherical spin glasses at low
  temperature: the perturbative regime.
\newblock {\em Comm. Pure Appl. Math.}, 73(8):1732--1828, 2020.

\bibitem[AT07]{AdlerTaylor}
R.~J. Adler and J.~E. Taylor.
\newblock {\em Random fields and geometry}.
\newblock Springer Monographs in Mathematics. Springer, New York, 2007.

\bibitem[AW09]{AzaisWscheborbook}
Jean-Marc Aza\"{\i}s and Mario Wschebor.
\newblock {\em Level sets and extrema of random processes and fields}.
\newblock John Wiley \& Sons, Inc., Hoboken, NJ, 2009.

\bibitem[BFG{\etalchar{+}}16]{Bloomfield:2016civ}
Jolyon~K. Bloomfield, Stephen H.~P. Face, Alan~H. Guth, Saarik Kalia, Casey
  Lam, and Zander Moss.
\newblock {Number Density of Peaks in a Chi-Squared Field}.
\newblock 12 2016.

\bibitem[Bil99]{Billingsley}
Patrick Billingsley.
\newblock {\em Convergence of probability measures}.
\newblock Wiley Series in Probability and Statistics: Probability and
  Statistics. John Wiley \& Sons, Inc., New York, second edition, 1999.
\newblock A Wiley-Interscience Publication.

\bibitem[BvNS22]{Belius22}
David Belius, Ji\v{r}\'{\i} \v{C}ern\'{y}, Shuta Nakajima, and Marius~A.
  Schmidt.
\newblock Triviality of the geometry of mixed {$p$}-spin spherical
  {H}amiltonians with external field.
\newblock {\em J. Stat. Phys.}, 186(1):Paper No. 12, 34, 2022.

\bibitem[CCF{\etalchar{+}}20]{Cheng2020}
Dan Cheng, Valentina Cammarota, Yabebal Fantaye, Domenico Marinucci, and Armin
  Schwartzman.
\newblock Multiple testing of local maxima for detection of peaks on the
  (celestial) sphere.
\newblock {\em Bernoulli}, 26(1):31--60, 2020.

\bibitem[CCM{\etalchar{+}}24]{Carones2024}
Alessandro {Carones}, Javier {Carr{\'o}nDuque}, Domenico {Marinucci}, Marina
  {Migliaccio}, and Nicola {Vittorio}.
\newblock {Minkowski functionals of CMB polarization intensity with PYNKOWSKI:
  theory and application to Planck and future data}.
\newblock {\em Monthly Notices of the Royal Astronomical Society},
  527(1):756--773, January 2024.

\bibitem[CMW16]{CMW16}
Valentina Cammarota, Domenico Marinucci, and Igor Wigman.
\newblock On the distribution of the critical values of random spherical
  harmonics.
\newblock {\em J. Geom. Anal.}, 26(4):3252--3324, 2016.

\bibitem[Col23]{LiteBird23}
LiteBIRD Collaboration.
\newblock Probing cosmic inflation with the litebird cosmic microwave
  background polarization survey.
\newblock {\em Progress of Theoretical and Experimental Physics 2023},
  (042F01), 2023.

\bibitem[CS17]{Cheng17}
Dan Cheng and Armin Schwartzman.
\newblock Multiple testing of local maxima for detection of peaks in random
  fields.
\newblock {\em Ann. Statist.}, 45(2):529--556, 2017.

\bibitem[FMM21]{Montanari21}
Zhou Fan, Song Mei, and Andrea Montanari.
\newblock T{AP} free energy, spin glasses and variational inference.
\newblock {\em Ann. Probab.}, 49(1):1--45, 2021.

\bibitem[FT22]{Fyodorov22}
Yan~V. Fyodorov and Rashel Tublin.
\newblock Optimization landscape in the simplest constrained random
  least-square problem.
\newblock {\em J. Phys. A}, 55(24):Paper No. 244008, 38, 2022.

\bibitem[Gay22]{GayetExEulMax}
Damien Gayet.
\newblock Asymptotic topology of excursion and nodal sets of gaussian random
  fields.
\newblock {\em Journal für die reine und angewandte Mathematik (Crelles
  Journal)}, 2022(790):149--195, 2022.

\bibitem[Hir94]{Hirsch}
Morris~W. Hirsch.
\newblock {\em Differential topology}, volume~33 of {\em Graduate Texts in
  Mathematics}.
\newblock Springer-Verlag, New York, 1994.
\newblock Corrected reprint of the 1976 original.

\bibitem[KM23]{Kuriki_Matsubara_2023}
Satoshi Kuriki and Takahiko Matsubara.
\newblock Asymptotic expansion of the expected minkowski functional for
  isotropic central limit random fields.
\newblock {\em Advances in Applied Probability}, 55(4):1390–1414, 2023.

\bibitem[LMRS22]{stec2022GeometrySpin}
Antonio Lerario, Domenico Marinucci, Maurizia Rossi, and Michele Stecconi.
\newblock Geometry and topology of spin random fields, 2022.

\bibitem[Mal13]{Malyarenko13}
Anatoliy Malyarenko.
\newblock {\em Invariant random fields on spaces with a group action}.
\newblock Probability and its Applications (New York). Springer, Heidelberg,
  2013.
\newblock With a foreword by Nikolai Leonenko.

\bibitem[MP11]{MarinucciPeccati11}
Domenico Marinucci and Giovanni Peccati.
\newblock {\em Random fields on the sphere}, volume 389 of {\em London
  Mathematical Society Lecture Note Series}.
\newblock Cambridge University Press, Cambridge, 2011.
\newblock Representation, limit theorems and cosmological applications.

\bibitem[MS24]{MathiStec}
L\'eo Mathis and Michele Stecconi.
\newblock Expectation of a random submanifold: the zonoid section.
\newblock {\em Annales Henri Lebesgue}, 7:903--967, 2024.

\bibitem[MSW69]{MilnorMorse}
J.~Milnor, M.~SPIVAK, and R.~WELLS.
\newblock {\em Morse Theory. (AM-51), Volume 51}.
\newblock Princeton University Press, 1969.

\bibitem[Nic17]{NICOLAESCU20173412}
Liviu~I. Nicolaescu.
\newblock A clt concerning critical points of random functions on a euclidean
  space.
\newblock {\em Stochastic Processes and their Applications},
  127(10):3412--3446, 2017.

\bibitem[PS24]{pistolato2024}
Francesca Pistolato and Michele Stecconi.
\newblock Expected lipschitz-killing curvatures for spin random fields and
  other non-isotropic fields, 2024.

\bibitem[Ste22]{KRStec}
Michele Stecconi.
\newblock Kac-{R}ice formula for transverse intersections.
\newblock {\em Analysis and Mathematical Physics}, 12(2):44, 2022.

\bibitem[TCPS23]{Telschow23}
Fabian J.~E. Telschow, Dan Cheng, Pratyush Pranav, and Armin Schwartzman.
\newblock Estimation of expected euler characteristic curves of nonstationary
  smooth random fields.
\newblock {\em Ann. Statist.}, 51(5):2272--2297, 2023.

\end{thebibliography}


\end{document}